\documentclass{amsart}

\usepackage{amssymb}
\usepackage{amsmath}
\usepackage{amsthm}
\usepackage[mathscr]{eucal}

\newtheorem{thrm}{Theorem}[section]
\newtheorem{lemma}[thrm]{Lemma}
\newtheorem{prop}[thrm]{Proposition}
\newtheorem{cor}[thrm]{Corollary}

\theoremstyle{definition}
\newtheorem{defn}[thrm]{Definition}

\theoremstyle{remark}

\numberwithin{equation}{section}

\newcommand{\mdbar}{\bar{\partial}}

\newcommand{\lre}{\mathscr{E}}
\newcommand{\lra}{\mathscr{A}}

\begin{document}

\bibliographystyle{plain}

\title[$C^k$ estimates on non-smooth domains]
{Weighted $C^k$ estimates for a class of integral operators
 on non-smooth domains}

\author{Dariush Ehsani}

\address{Department of Mathematics, Penn State - Lehigh Valley, Fogelsville,
PA 18051}
 \email{ehsani@psu.edu}
\curraddr{Humboldt-Universit\"{a}t, Institut f\"{u}r Mathematik,
10099 Berlin }


\thanks{Partially supported by the Alexander von Humboldt Stiftung}

\begin{abstract}
  We apply integral representations for
$(0,q)$-forms, $q\ge1$, on non-smooth strictly pseudoconvex
domains, the Henkin-Leiterer domains, to derive weighted $C^k$
estimates for a given $(0,q)$-form, $f$, in terms of $C^k$ norms
of $\mdbar f$, and $\mdbar^{\ast} f$.  The weights are powers of
the gradient of the defining function of the domain.
\end{abstract}

\maketitle
\ \\
\section{Introduction}
Let $X$ be an $n$-dimensional complex manifold, equipped with a
Hermitian metric, and $D\subset\subset X$ a strictly pseudoconvex
domain with defining function $r$.  Here we do not assume the
non-vanishing of the gradient, $dr$, thus allowing for the
possibility of singularities in the boundary, $\partial D$ of $D$.
 We refer to such domains as Henkin-Leiterer
domains, as they were first systematically studied by Henkin and
Leiterer in \cite{HL}.

  We shall make the additional assumtion that $r$ is a Morse
function.

Let $\gamma=|\partial r|$. In \cite{Eh09a} the author established
an integral representation of the form
\begin{thrm}
\label{intrthrm}   There exist integral operators
$\tilde{T}_q:L^2_{(0,q+1)}(D)\rightarrow L^2_{(0,q)}(D)$ $0\le
q<n=\dim X$ such that for $f\in L^2_{(0,q)}\cap
Dom(\bar{\partial})\cap Dom(\bar{\partial}^{\ast})$ one has
\begin{equation*}
\gamma^3
 f=\tilde{T}_q\bar{\partial}f+\tilde{T}_{q-1}^{\ast}\bar{\partial}^{\ast}f
+\mbox{ error terms } \quad \mbox{ for } q\ge 1.
\end{equation*}
\end{thrm}
Theorem \ref{intrthrm} is valid under the assumption we are
working with the Levi metric.  With local coordinates denoted by
$\zeta_1,\ldots,\zeta_n$, we define a Levi metric in a
neighborhood of $\partial D$ by
\begin{equation*}
ds^2= \sum_{j,k} \frac{\partial^2 r}{\partial \zeta_j, \partial
\overline{\zeta}_k} (\zeta).
\end{equation*}
A Levi metric on $X$ is a Hermitian metric which is a Levi metric
in a neighborhood of $\partial D$.  From what follows we will be
working with $X$ equipped with a Levi metric.

The author then used properties of the operators in the
representation to establish the estimates
\begin{thrm} For $f\in L^2_{0,q}(D)\cap\mbox{Dom}(\mdbar)\cap
\mbox{Dom}(\mdbar^{\ast})$, $q\ge 1$, \label{k=0}
\begin{equation*}
\|\gamma^{3(n+1)}f\|_{L^{\infty}}\lesssim \|\gamma^2\mdbar
f\|_{\infty}+\|\gamma^2\mdbar^{\ast}f\|_{\infty}+\|f\|_2 .
\end{equation*}
\end{thrm}

In this paper, we examine the operators in the integral
representation, derive more detailed properties of such operators
under differentiation, and use the properties to establish $C^k$
estimates.  Our main theorem is
\begin{thrm}
 \label{ourthrm}
 Let $f\in L^2_{0,q}(D)\cap\mbox{Dom}(\mdbar)\cap
\mbox{Dom}(\mdbar^{\ast})$, $q\ge 1$, and $\alpha<1/4$.  Then for
$N(k)$ large enough we have
\begin{equation*}
\|\gamma^{N(k)} f\|_{C^{k+\alpha}}\lesssim
 \|\gamma^{k+2}
\mdbar f\|_{C^{k}}+\|\gamma^{k+2}
\mdbar^{\ast}f\|_{C^{k}}+\|f\|_{2}.
\end{equation*}
\end{thrm}
We show we may take any $N(k)>3(n+6)+8k$.

Our results are consistent with those obtained by Lieb and Range
in the case of smooth strictly pseudoconvex domains \cite{LR86_2},
where we may take $\gamma=1$.  In, \cite{LR86_2}, an estimate as
in Theorem \ref{ourthrm} with $\gamma=1$ and $\alpha<1/2$ was
given.

In a separate paper we look establish $C^k$ estimates for $f\in
L^2(D)\cap\mbox{Dom}(\mdbar)$, as the functions used in the
construction of the integral kernels in the case $q=0$ differ from
those in the case $q\ge1$.

One of the difficulties in working on non-smooth domains is the
problem of the choice of frame of vector fields with which to
work.  In the case of smooth domains a special boundary chart is
used in which $\omega^n=\partial r$ is part of an orthonormal
frame of $(1,0)$-forms.  When $\partial r$ is allowed to vanish,
the frame needs to be modified.  We get around this difficulty by
defining a $(1,0)$-form, $\omega^n$ by $\partial r = \gamma
\omega^n$.  In the dual frame of vector fields we are then faced
with factors of $\gamma$ in the expressions of the vector fields
with respect to local coordinates, and we deal with these terms by
multiplying our vector fields by a factor of $\gamma$.  This
ensures that when vector fields are commuted, there are no error
terms which blow up at the singularity.

We organize our paper as follows.  In Section \ref{adms} we define
the types of operators which make up the integral representation
established in \cite{Eh09a}.  Section \ref{est} contains the most
essential properties used to obtain our results.  In Section
\ref{est} we consider the properties of our integral operators
under differentiation.  Lastly, in Section \ref{ck} we apply the
properties from Section \ref{est} to obtain our $C^k$ estimates.

The author extends thanks to Ingo Lieb with whom he shared many
fruitful discussions over the ideas presented here, and from whom
he originally had the idea to extend results on smooth domains to
Henkin-Leiterer domains.

\section{Admissible operators}
 \label{adms}
  With local coordinates denoted by $\zeta_1,\ldots,\zeta_n$, we
define a Levi metric in a neighborhood of $\partial D$ by
\begin{equation*}
ds^2= \sum_{j,k} \frac{\partial^2 r}{\partial \zeta_j, \partial
\overline{\zeta}_k} (\zeta)d\zeta_j d\bar{\zeta}_k.
\end{equation*}
A Levi metric on $X$ is a Hermitian metric which is a Levi metric
in a neighborhood of $\partial D$.

We thus equip $X$ with a Levi metric and we take $\rho(x,y)$ to be
a symmetric, smooth function on $X\times X$ which coincides with
the geodesic distance in a neighborhood of the diagonal,
$\Lambda$, and is positive outside of $\Lambda$.

 For ease of
notation, in what follows we will always work with local
coordinates, $\zeta$ and $z$.

      Since $D$ is strictly pseudoconvex and $r$ is a Morse function, we
can take $r_{\epsilon}=r+\epsilon$ for epsilon small enough.  Then
$r_{\epsilon}$ will be defining functions for smooth, strictly
pseudoconvex $D_{\epsilon}$.  For such $r_{\epsilon}$ we have that
all derivatives of $r_{\epsilon}$ are indpendent of $\epsilon$.
 In particular, $\gamma_{\epsilon}(\zeta)=\gamma(\zeta)$ and
$\rho_{\epsilon}(\zeta,z)=\rho(\zeta,z)$.

 Let $F$ be the Levi
polynomial for $D_{\epsilon}$:
\begin{equation*}
F(\zeta,z)
 = \sum_{j=1}^n\frac{\partial
 r_{\epsilon}}{\partial\zeta_j}(\zeta)(\zeta_j-z_j)
  -\frac{1}{2}\sum_{j,k=1}^n\frac{\partial^2
  r_{\epsilon}}{\partial\zeta_j\zeta_k}(\zeta_j-z_j)(\zeta_k-z_k).
\end{equation*}
 We note that $F(\zeta,z)$ is independent of $\epsilon$ since
 derivatives of $r_{\epsilon}$ are.

For $\epsilon$ small enough we can choose $\delta>0$ and
$\varepsilon>0$ and a patching function $\varphi(\zeta,z)$,
independent of $\epsilon$, on $\mathbb{C}^n\times\mathbb{C}^n$
such that
\begin{equation*}
\varphi(\zeta,z)=
 \begin{cases}
1 & \mbox{for } \rho^2(\zeta,z)\le
\frac{\varepsilon}{2}\\
0 &  \mbox{for } \rho^2(\zeta,z)\ge \frac{3}{4}\varepsilon,
\end{cases}
\end{equation*}
and defining $S_{\delta}=\{ \zeta:|r(\zeta)|<\delta\}$,
$D_{-\delta}=\{ \zeta:r(\zeta)<\delta\}$, and
\begin{equation*}
\phi_{\epsilon}(\zeta,z)
 =\varphi(\zeta,z)(F_{\epsilon}(\zeta,z)-r_{\epsilon}(\zeta))
  +(1-\varphi(\zeta,z))\rho^2(\zeta,z),
 \end{equation*}
we have the following
\begin{lemma}
 \label{phiest}
On $D_{\epsilon}\times D_{\epsilon}\bigcap S_{\delta}\times
D_{-\delta}$,
 \begin{equation*}
|\phi_{\epsilon}|\gtrsim |\langle \partial
r_{\epsilon}(z),\zeta-z\rangle|+\rho^2(\zeta,z),
\end{equation*}
where the constants in the inequalities are independent of
$\epsilon$.
\end{lemma}

  We at times have to be precise and keep track of factors
 of $\gamma$ which occur in our integral kernels.  We shall write
$\lre_{j,k}(\zeta,z)$ for those double forms on open sets
$U\subset D\times D$ such that $\lre_{j,k}$ is smooth on $U$ and
satisfies
\begin{equation}
 \label{defnxi}
\lre_{j,k}(\zeta,z)\lesssim \xi_k(\zeta) |\zeta-z|^j,
\end{equation}
 where
 $\xi_k$ is a smooth function in $D$ with the property
\begin{equation*}
|\gamma^{\alpha}D_{\alpha}\xi_k|\lesssim \gamma^k,
\end{equation*}
for $D_{\alpha}$ a differential operator of order $\alpha$.

 We shall write
 $\lre_j$ for those double forms on open sets
$U\subset D\times D$ such that $\lre_{j}$ is smooth on $U$, can be
extended smoothly to $\overline{D}\times\overline{D}$, and
satisfies
\begin{equation*}
\lre_{j}(\zeta,z)\lesssim  |\zeta-z|^j.
\end{equation*}
$\lre_{j,k}^{\ast}$ will
 denote
 forms which can be written as $\lre_{j,k}(z,\zeta)$.

For $N\ge 0$, we let $R_N$ denote an $N$-fold product, or a sum of
such products, of first derivatives of $r(z)$, with the notation
$R_0=1$.

 Here
\begin{equation*}
P_{\epsilon}(\zeta,z)=\rho^2(\zeta,z)+
\frac{r_{\epsilon}(\zeta)}{\gamma(\zeta)}\frac{r_{\epsilon}(z)}{\gamma(z)}.
\end{equation*}

\begin{defn} A double differential form $\lra^{\epsilon}(\zeta,z)$ on
$\overline{D}_{\epsilon}\times\overline{D}_{\epsilon}$ is an
\textit{admissible} kernel, if it has the following properties:
\begin{enumerate}
\item[i)] $\lra^{\epsilon}$ is smooth on
$\overline{D}_{\epsilon}\times\overline{D}_{\epsilon}-\Lambda_{\epsilon}$
 \item[ii)] For each point $(\zeta_0,\zeta_0)\in \Lambda_{\epsilon}$ there is
 a neighborhood $U\times U$ of $(\zeta_0,\zeta_0)$ on which $\lra^{\epsilon}$ or $\overline{\lra}^{\epsilon}$
 has the representation
 \begin{equation}
 \label{typerep}
  R_NR_M^* \lre_{j,\alpha}\lre_{k,\beta}^{\ast}
  P^{-t_0}_{\epsilon}\phi^{t_1}_{\epsilon}\overline{\phi}^{t_2}_{\epsilon}
   \phi^{\ast t_3}_{\epsilon}\overline{\phi}^{\ast t_4}_{\epsilon} r^l_{\epsilon} r^{\ast
   m}_{\epsilon}
 \end{equation}
with $N,M,\alpha,\beta,j,k, t_0, \ldots, m$ integers and $j,k,
t_0, l, m \ge 0$,
 $-t=t_1+\cdots+t_4\le 0$, $N, M\ge
0$, and $N+\alpha, M+\beta\ge 0$.
\end{enumerate}
The above representation is of \textit{smooth type} $s$ for
\begin{equation*}
s=2n+j+\min\{2, t-l-m\} -2(t_0+t-l-m).
\end{equation*}
We define the \textit{type} of $\lra^{\epsilon}(\zeta,z)$ to be
\begin{equation*}
\tau=s-\max \{ 0,2-N-M-\alpha-\beta\}.
\end{equation*}
  $\lra^{\epsilon}$ has \textit{smooth
type} $\ge s$ if at each point $(\zeta_0,\zeta_0)$ there is a
representation (\ref{typerep}) of smooth type $\ge s$.
$\lra^{\epsilon}$ has \textit{type} $\ge \tau$ if at each point
$(\zeta_0,\zeta_0)$ there is a representation (\ref{typerep}) of
type $\ge \tau$.  We shall also refer to the \textit{double type}
of an operator $(\tau,s)$ if the operator is of type $\tau$ and of
smooth type $s$.
\end{defn}
The definition of smooth type above is taken from \cite{LR86}.
Here and below $(r_{\epsilon}(x))^{\ast }=r_{\epsilon}(y)$, the
$\ast$ having a similar meaning for other functions of one
variable.

Let $\lra_{j}^{\epsilon}$ be kernels of type $j$.  We denote by
$\lra_j$ the pointwise limit as $\epsilon\rightarrow 0$ of
$\lra_j^{\epsilon}$ and define the double type of $\lra_j$ to be
the double type of the $\lra_j^{\epsilon}$ of which it is a limit.
We also denote by $A_j^{\epsilon}$ to be operators with kernels of
the form $\lra_j^{\epsilon}$.  $A_j$ will denote the operators
with kernels $\lra_j$.  We use the notation
$\lra_{(j,k)}^{\epsilon}$ (resp. $\lra_{(j,k)}$) to denote kernels
of double type $(j,k)$.

 We let $\lre_{j-2n}^i(\zeta,z)$ be a
kernel of the form
\begin{equation*}
\lre_{j-2n}^i(\zeta,z)=
 \frac{\lre_{m,0}(\zeta,z)}{\rho^{2k}(\zeta,z)} \qquad j\ge 1,
\end{equation*}
where $m-2k\ge j-2n$. We denote by $E_{j-2n}$ the corresponding
isotropic operator.

 From \cite{Eh09a}, we have
\begin{thrm}
\label{basicintrep}
 For $f\in L^2_{(0,q)}(D)\cap Dom(\mdbar)\cap Dom(\mdbar^{\ast})$,
 there exist integral operators $T_q$, $S_q$, and $P_q$ such that
\begin{equation*}
\gamma(z)^3 f(z)=\gamma^{\ast} T_q \mdbar \left(\gamma^2 f\right)
+\gamma^{\ast} S_q\mdbar^{\ast}\left(\gamma^2 f\right)
+\gamma^{\ast} P_q\left(\gamma^2 f\right).
\end{equation*}
$T_q$, $S_q$, and $P_q$ have the form
\begin{align*}
&T_q=E_{1-2n}+A_1\\
&S_q=E_{1-2n}+A_1\\
&P_q=\frac{1}{\gamma}A_{(-1,1)}^{\epsilon}+
\frac{1}{\gamma^{\ast}}A_{(-1,1)}^{\epsilon}
\end{align*}
\end{thrm}

\section{Estimates}
\label{est}

 We begin with estimates on the kernels of a certain type.  In
 \cite{Eh09a}, we proved the
\begin{prop}
 \label{yngcor}
Let $A_j^{}$ be an operator of type j. Then
\begin{equation*}
A_j^{}:L^p(D)\rightarrow L^s(D) \quad
\frac{1}{s}>\frac{1}{p}-\frac{j}{2n+2}.
\end{equation*}
\end{prop}

 We describe what we shall call
tangential derivatives on the Henkin-Leiterer domain, $D$.  A
non-vanishing vector field, $T$, in $\mathbb{R}^{2n}$ will be
called tangential if $Tr=0$ on $r=0$.  Near a boundary point, we
choose a coordinate patch on which we have an orthogonal frame
$\omega^1,\ldots, \omega^n$ of $(1,0)$-forms with $\partial
r=\gamma\omega^n$.  Let $L_1,\ldots,L_n$ denote the dual frame.
$L_1,\ldots, L_{n-1}$, $\overline{L}_1,\ldots,\overline{L}_{n-1}$,
and $Y=L_n-\overline{L}_n$ are tangential vector fields.
$N=L_n+\overline{L}_n$ is a normal vector field.  We say a given
vector field $X$ is a smooth tangential vector field if it is a
tangential field and if near each boundary point $X$ is a
combination of such vector fields $L_1,\ldots, L_{n-1}$,
$\overline{L}_1,\ldots,\overline{L}_{n-1},Y$, and $rN$ with
coefficients in $C^{\infty}(\overline{D})$.  We make the important
remark here that in the coordinate patch of a critical point, the
smooth tangential vector fields are not smooth combinations of
derivatives with respect to the coordinate system described in
Lemma \ref{c1diff}.  In fact, they are combinations of derivatives
with respect to the coordinates of Lemma \ref{c1diff} with
coefficients only in $C^0(\overline{D})$ due to factors of
$\gamma$ which occur in the denominators of such coefficients.  In
general a $k^{th}$ order derivative of such coefficients is in
$\lre_{0,-k}$. Thus, when integrating by parts, special attention
has to be paid to these non-smooth terms.
\begin{defn}  We say an operator with kernel, $\lra$, is of
\textit{commutator} type $j$ if $\lra$ is of type $j$, and if in
the representation of $\lra$ in (\ref{typerep}) we have $t_1t_3\ge
0$, $t_2t_4\ge0$, and $(t_1+t_3)(t_2+t_4)\le0$.
\end{defn}
\begin{defn} Let $W$ be a smooth tangential vector field
on $\overline{D}$.  We call $W$ \textit{allowable} if for all
$\zeta\in\partial D$
\begin{equation*}
W^{\zeta}\in T_{\zeta}^{1,0}(\partial D)\oplus
T_{\zeta}^{0,1}(\partial D).
\end{equation*}
\end{defn}

 The following theorem is
obtained by a modification of Theorem 2.20 in \cite{LR86_2} (see
also \cite{LiMi}).  The new details which come about from the fact
that here we do not assume $|\partial r|\ne 0$ require careful
consideration and so we go through the calculations below.
\begin{thrm}
\label{commutator}
 Let $A_1$ be an admissible operator of
 commutator type $\ge 1$ and $X$ a smooth tangential vector
field. Then
\begin{equation*}
\gamma^{\ast}
 X^{z}A_1 =-A_1\tilde{X}^{\zeta}\gamma+A_1^{(0)} +\sum_{\nu=1}^l
A_1^{(\nu)}  W_{\nu}^{\zeta}\gamma ,
\end{equation*}
where $\tilde{X}$ is the adjoint of $X$, the $W_{\nu}$ are
allowable vector fields, and the $A_j^{(\nu)}$ are admissible
operators of commutator type $\ge j$.
\end{thrm}
\begin{proof}
We use a partition of unity and suppose $X$ has arbitrarily small
support on a coordinate patch near a boundary point in which we
have an orthogonal frame $\omega^1, \ldots, \omega^n$ of
$(1,0)$-forms with $\partial r=\gamma\omega^n$, as described above
with $L_1, \dots, L_n$ comprising the dual frame. We have
$L_1,\ldots, L_{n-1}$, $\overline{L}_1,\ldots,\overline{L}_{n-1}$,
and $Y=L_n-\overline{L}_n$ as tangential vector fields, and
$N=L_n+\overline{L}_n$ a normal vector field.

We have the decomposition of the tangential vector field $X$
\begin{equation*}
X=\sum_{j=0}^{n-1}a_j L_j + \sum_{j=0}^{n-1}b_j \overline{L}_j
 +aY+brN,
\end{equation*}
where the $a_j$, $b_j$, $a$, and $b$ are smooth with compact
support.  We then prove the theorem for each term in the
decomposition.
\ \\
$Case\ 1)$.  $X=a_jL_j$ or $b_j\overline{L}_j$, $j\le n-1$, or
$aY$.

We write
\begin{equation*}
\gamma^{\ast}X^{z}\lra_{1} = -\gamma X^{\zeta}\lra_{1}+
 (\gamma X^{\zeta}+\gamma^{\ast} X^{z})\lra_{1}.
\end{equation*}
Then an integration by parts gives
\begin{equation*}
\gamma^{\ast}X^zA_1  f=
 -A_1(\widetilde{X}^{\zeta}\gamma f)+(f,(\gamma X^{\zeta}+\gamma^{\ast} X^z)\lra_1).
\end{equation*}

We now use the following relations
\begin{align}
 \label{allreln}
(\gamma X^{\zeta}+\gamma^{\ast}X^z)\lre_{j,\alpha}&=\lre_{j,\alpha}\\
\nonumber
(\gamma X^{\zeta}+\gamma^{\ast}X^z)\lre_{j,\beta}^{\ast}&=\lre_{j,\beta}^{\ast}\\
\nonumber
 (\gamma X^{\zeta}+\gamma^{\ast}X^z)P&= \lre_{2,0}
+\frac{rr^{\ast}}{\gamma\gamma^{\ast}}\lre_{0,0}
 \\
  \nonumber
&=\lre_{0,0} P+\lre_{2,0}\\
\nonumber
 (\gamma X^{\zeta}+\gamma^{\ast}X^z)\phi &=\lre_{1,1}+\lre_{2,0}.
\end{align}

Any type 1 kernel
\begin{equation}
\label{regtype}
 \lra_1(\zeta,z)=R_NR_M^{\ast}\lre_{j,\alpha}\lre_{k,\beta}^{\ast}
 P^{-t_0}\phi^{t_1}
 \overline{\phi}^{t_2}
   \phi^{\ast t_3}\overline{\phi}^{\ast t_4} r^l r^{\ast
   m}
\end{equation}
can be decomposed into terms
\begin{equation*}
\lra_1=\lra_1'+\lra_2
\end{equation*}
where $\lra_1'$ is of \textit{pure} type, meaning it has a
representation as in (\ref{regtype}) but with $t_3=t_4=0$ and
$t_1t_2\le0$, \cite{LR86_2}.
 From the relations (\ref{allreln}) we have
\begin{equation*}
(\gamma X^{\zeta}+\gamma^{\ast}X^z)\lra_2 = \gamma\lra_1+\lra_2.
\end{equation*}

In calculating $ (\gamma X^{\zeta}+\gamma^{\ast} X^z)\lra_1'$, we
find the term that is not immediately seen to be of type $\lra_1$
is that which results from the operator $\gamma
X^{\zeta}+\gamma^{\ast}X^z$ falling on $\phi^{t_1}$, in which case
we obtain the term of double type $(0,0)$
\begin{equation*}
B:= R_NR_M^{\ast}\lre_{j+1,\alpha+1}\lre_{k,\beta}^{\ast}
 P^{-t_0}\phi^{t_1-1}\overline{\phi}^{t_2}r^l
r^{\ast m},
\end{equation*}
where $N+\alpha\ge2$, plus a term which is $\lra_1$.
 We follow \cite{LiMi} to reduce to the case where $B$ can be
written as a sum of terms $B_{\sigma}$ such that $B_{\sigma}$ or
$\overline{B}_{\sigma}$ is of the form
\begin{equation*}
\gamma^2\phi^{\sigma}(\phi+\overline{\phi})^{\tau_1+\tau_2-\sigma}
 R_NR_M^{\ast}\lre_{j+1,\alpha-1}\lre_{k,\beta}^{\ast}P^{-t_0}
r^lr^{\ast m},
\end{equation*}
where $\tau_1+\tau_2\le-3$ and
 $\tau_1\le \sigma\le \tau_1+\tau_2$ or $\tau_2\le \sigma \le
 \tau_1+\tau_2$.

We fix a point $z$ and choose local coordinates $\zeta$ such that
\begin{equation*}
 d\zeta_j(z)=\omega_j(z).
\end{equation*}
 Working in a neighborhood of a singularity in the boundary
(where we can use a coordinate system as in (\ref{rcoor}) below),
we see
 $\frac{\partial}{\partial \zeta_n}$ is a combination of
 derivatives with coefficients of the form $\xi_0(z)$, while
 $L_n$ is a combination of
 derivatives with coefficients of the form $\xi_0(\zeta)$, where
 $\xi_0$ is defined in (\ref{defnxi}).
  We have $\Lambda_n -\frac{\partial}{\partial z_n}$ is a sum of
  terms of the form
\begin{equation*}
(\xi_0(z)-\xi_0(\zeta))\Lambda^{\epsilon} =
\lre_{1,-1}\Lambda^{\epsilon},
\end{equation*}
where $\Lambda$ is a first order differential operator, and the
equality follows from
\begin{align*}
\frac{1}{\gamma(\zeta)}-\frac{1}{\gamma(z)}&
 =\frac{\gamma(z)-\gamma(\zeta)}{\gamma(\zeta)\gamma(z)}\\
&=\frac{1}{\gamma(z)}\frac{\gamma^2(z)-\gamma^2(\zeta)}
 {\gamma(\zeta)(\gamma(\zeta)+\gamma(z))}\\
&=\frac{1}{\gamma(z)}\frac{\xi_1(\zeta)\lre_1}
 {\gamma(\zeta)(\gamma(\zeta)+\gamma(z))}\\
&=\frac{1}{\gamma(z)}\frac{\lre_{1,0}}
 {(\gamma(\zeta)+\gamma(z))}\\
&\lesssim \frac{1}{\gamma(z)}\frac{\lre_{1,0}}
 {\gamma(z)}\\
&=\lre_{1,-2}.
\end{align*}
Using these special coordinates, we note
\begin{align*}
Y\phi&=\gamma+\lre_{1,0}+\lre_{2,-1}\\
Y\overline{\phi}&=-\gamma+\lre_{1,0}+\lre_{2,-1}\\
YP&= \lre_{1,0}+\frac{\lre_{0,0}}{\gamma}\left(P+\lre_{2,0}\right)
\end{align*}
 and write
\begin{align*}
B_{\sigma}=&\gamma^2\phi^{\sigma}(\phi+\overline{\phi})^{\tau_1+\tau_2-\sigma}R_NR_M^{\ast}
 \lre_{j+1,\alpha-1}\lre_{k,\beta}^{\ast}P^{-t_0}
r^lr^{\ast m}\\
=&\gamma
Y\left(\phi^{\sigma+1}(\phi+\overline{\phi})^{\tau_1+\tau_2-\sigma}R_NR_M^{\ast}
 \lre_{j+1,\alpha-1}\lre_{k,\beta}^{\ast}P^{-t_0}
r^lr^{\ast m}\right)\\
&+\gamma\phi^{\sigma}(\phi+\overline{\phi})^{\tau_1+\tau_2-\sigma}R_NR_M^{\ast}
 \lre_{j+2,\alpha-1}\lre_{k,\beta}^{\ast}P^{-t_0}
r^lr^{\ast m}\\
&+\gamma\phi^{\sigma}(\phi+\overline{\phi})^{\tau_1+\tau_2-\sigma}R_NR_M^{\ast}
 \lre_{j+3,\alpha-2}\lre_{k,\beta}^{\ast}P^{-t_0}
r^lr^{\ast m}\\
&+\gamma\phi^{\sigma+1}(\phi+\overline{\phi})^{\tau_1+\tau_2-\sigma-1}R_NR_M^{\ast}
 \lre_{j+2,\alpha-1}\lre_{k,\beta}^{\ast}P^{-t_0}
r^lr^{\ast m}\\
&+\gamma\phi^{\sigma+1}(\phi+\overline{\phi})^{\tau_1+\tau_2-\sigma-1}R_NR_M^{\ast}
 \lre_{j+3,\alpha-2}\lre_{k,\beta}^{\ast}P^{-t_0}
r^lr^{\ast m}\\
&+\gamma\phi^{\sigma+1}(\phi+\overline{\phi})^{\tau_1+\tau_2-\sigma}R_{N-1}R_M^{\ast}
 \lre_{j+1,\alpha-1}\lre_{k,\beta}^{\ast}P^{-t_0}
r^lr^{\ast m}\\
&+\gamma\phi^{\sigma+1}(\phi+\overline{\phi})^{\tau_1+\tau_2-\sigma}R_{N}R_M^{\ast}
 \lre_{j,\alpha-1}\lre_{k,\beta}^{\ast}P^{-t_0}
r^lr^{\ast m}\\
&+\gamma\phi^{\sigma+1}(\phi+\overline{\phi})^{\tau_1+\tau_2-\sigma}R_{N}R_M^{\ast}
 \lre_{j+2,\alpha-1}\lre_{k,\beta}^{\ast}P^{-t_0-1}
r^lr^{\ast m}\\
&+\phi^{\sigma+1}(\phi+\overline{\phi})^{\tau_1+\tau_2-\sigma}R_NR_M^{\ast}
 \lre_{j+1,\alpha-1}\lre_{k,\beta}^{\ast}P^{-t_0}
r^lr^{\ast m}\\
&+\phi^{\sigma+1}(\phi+\overline{\phi})^{\tau_1+\tau_2-\sigma}R_{N}R_M^{\ast}
 \lre_{j+3,\alpha-1}\lre_{k,\beta}^{\ast}P^{-t_0-1}
r^lr^{\ast m}.
\end{align*}
Thus
\begin{equation*}
B_{\sigma}=\gamma Y\lra_{(1,2)}+\lra_{1}'.
\end{equation*}

By the strict pseudoconvexity of $D$ there exists allowable vector
fields $W_1$, $W_2$, and $W_3$, and a function $\varphi$, smooth
on the interior of $D$ which satisfies
\begin{equation*}
\Phi^k\varphi= \lre_{0,1-k},
\end{equation*}
where $\Phi$ is a first order differential operator,
 such that $Y$ can be
written
\begin{equation*}
Y= \varphi[W_1,W_2]+W_3.
\end{equation*}
Thus
\begin{align*}
\gamma Y\lra_{(2,1)}&=
 \gamma\varphi[W_1,W_2]\lra_{(1,2)}+\gamma W_3\lra_{(1,2)}\\
 &=\gamma
 [W_1,W_2]\varphi\lra_{(1,2)}+\lra_{1}'
\end{align*}
with $\lra_1'$ of commutator type $\ge 1$.

An integration by parts gives
\begin{equation*}
 (f,\gamma[W_1,W_2](\varphi\lra_2))=
  (\widetilde{W}_1\gamma f,W_2(\varphi\lra_{(1,2)}))-
  (\widetilde{W}_2\gamma f,W_1(\varphi\lra_{(1,2)})).
 \end{equation*}
$\widetilde{W}_1$ and $\widetilde{W}_2$ are allowable vector
fields while $W_2(\varphi\lra_{(1,2)})$ and
$W_1(\varphi\lra_{(1,2)})$ are of the form
 $\lra_{1}'$ where $\lra_1'$ is of commutator type,
 proving the theorem for
Case 1.
\\
$Case\ 2)$.  $X=\lre_0 rN$.

We use
\begin{align}
 \label{case2reln}
\left( r\gamma N^{\zeta}+r^{\ast}\gamma^{\ast}
 N^z\right)\lre_{j,\alpha}
&=\lre_{j,\alpha}\\
 \nonumber
\left( r\gamma N^{\zeta}+r^{\ast}\gamma^{\ast}
N^z\right)P&=\lre_{2,0}+
 \frac{r}{\gamma}\frac{r^{\ast}}{\gamma^{\ast}}\lre_{0,0}\\
 \nonumber
 &=\lre_{2,0}+
 P\lre_{0,0}\\
 \nonumber
\left(r\gamma N^{\zeta}+r^{\ast}\gamma^{\ast}N^z\right)\phi&=
 r\lre_{0,0}+r^{\ast}\lre_{0,0}.
\end{align}
Thus
\begin{align*}
\gamma^{\ast}XA_1 f=&
(\lre_0  r^{\ast}  f,\gamma^{\ast}N^{z}\lra_1)\\
=& (-\lre_0  r  f,\gamma N^{\zeta}\lra_1)+( f,\lre_0\left(
r\gamma N^{\zeta}+r^{\ast}\gamma^{\ast}N^z\right)\lra_1)\\
=& (-\widetilde{N}^{\zeta}(\lre_0 r \gamma
f),\lra_1)+(f,\lre_0\left( r\gamma
N^{\zeta}+r^{\ast}\gamma^{\ast}N^z\right)\lra_1).
\end{align*}
We have
\begin{equation*}
\widetilde{N}^{\zeta}(\lre_0 r \gamma f)
 = \lre_{0,0}  f+\lre_0r\widetilde{N}^{\zeta}\gamma f
\end{equation*}
and $\lre_0r\widetilde{N}^{\zeta}$ is an allowable vector field.
The relations in (\ref{case2reln}) show that
\begin{equation*}
\left( r\gamma N^{\zeta}+r^{\ast}\gamma^{\ast} N^z\right)\lra_1
\end{equation*}
 is of commutator type $\ge 1$. Case 2 therefore follows.
\end{proof}

Below we use a criterion for H\"{o}lder continuity given by
Schmalz (see Lemma 4.1 in \cite{Sc89}) which states
\begin{lemma}[Schmalz]
\label{schmlemma} Let $D\subseteq \mathbb{R}^m$, $m\ge 1$ be an
open set and let $B(D)$ denote the space of bounded functions on
$D$. Suppose $r$ is a $C^2$ function on $\mathbb{R}^m$, $m\ge 1$,
such that $D:=\{r<0\} \subseteq \mathbb{R}^m$.  Then there exists
a constant $C<\infty$ such that the following holds: If a function
$u\in B(D)$ satisfies for some $0<\alpha\le 1/2$ and for all
$z,w\in D$ the estimate
\begin{equation*}
|u(z)-u(w)|\le |z-w|^{\alpha}+\max_{y=z,w}
 \frac{|\nabla r(y)| |z-w|^{1/2+\alpha}}{|r(y)|^{1/2}}
\end{equation*}
then
\begin{equation*}
|u(z)-u(w)|\le C|z-w|^{\alpha}
\end{equation*}
for all $z,w\in D$.
\end{lemma}

 We will also refer to a lemma of
Schmalz (Lemma 3.2 in \cite{Sc89}) which provides a useful
coordinate system in which to prove estimates.
\begin{lemma}
 \label{schmalzlem}
 Define $x_j$ by $\zeta_j = x_j+ix_{j+n}$ for $1\le j\le n$.
Let $E_{\delta}(z):=\{\zeta\in D: |\zeta-z| < \delta \gamma(z)\}$
for $\delta>0$.  Then there is a constant $c$ and numbers $l$,
$m\in \{1,\ldots,2n\}$ such that for all $z\in D$,
 \begin{equation*}
 \{
-r(\zeta),\mbox{
Im}\phi(\cdot,z),x_1,\ldots_{\hat{l}},_{\hat{m}}\ldots,x_{2n}\},
\end{equation*}
where $x_l$ and $x_m$ are omitted, forms a coordinate system in
$E_c(z)$ .  We have the estimate
\begin{equation*}
\label{volest}
 dV \lesssim \frac{1}{\gamma(z)^2}\left|
dr(\zeta)\wedge d\mbox{ Im}\phi(\cdot,z)\wedge dx_1\wedge
\ldots_{\hat{l}},_{\hat{m}}\ldots\wedge dx_{2n}\right| \quad
\mbox{on } E_c(z),
\end{equation*}
where $dV$ is the Euclidean volume form on $\mathbb{R}^{2n}$.
\end{lemma}

We define the function spaces with which we will be working.
\begin{defn}  Let $0\le\beta$ and $0\le\delta$.  We define
\begin{equation*}
\|f\|_{L^{\infty,\beta,\delta}(D)}=\sup_{\zeta\in D}
 |f(\zeta)|\gamma^{\beta}(\zeta)|r(\zeta)|^{\delta}.
\end{equation*}
\end{defn}
\begin{defn}
We set for $0<\alpha<1$
\begin{equation*}
\Lambda_{\alpha}(D)
 =\{ f\in L^{\infty}(D)\ |\ \|f\|_{\Lambda_{\alpha}}
 := \|f\|_{L^{\infty}}+\sup
 \frac{|f(\zeta)-f(z)|}{|\zeta-z|^{\alpha}}<\infty \}.
\end{equation*}
\end{defn}
We also define the spaces $\Lambda_{\alpha,\beta}$ by
\begin{equation*}
 \Lambda_{\alpha,\beta}:=\{ f: \|f\|_{\Lambda_{\alpha,\beta}}=
\|\gamma^{\beta}f\|_{\Lambda_{\alpha}}<\infty \}.
\end{equation*}

From \cite{Eh09a}, we have the
     \begin{lemma}
\label{c1diff}
\begin{equation*}
\frac{r_{\epsilon}}{\gamma}\in C^1(D_{\epsilon})
\end{equation*}
with $C^1$-estimates independent of $\epsilon$.
\end{lemma}

For our $C^k$ estimates later, we will need the following
properties.
\begin{thrm}
\label{dertype1}
 Let $T$ be a smooth first order tangential differential operator on $D$.
 For $A$ an operator of type 1 we have
\begin{align*}
&i)\ A:L^{\infty,2+\epsilon,0}(D)\rightarrow
 \Lambda_{\alpha,2-\epsilon'}(D)
\qquad 0<\epsilon,\epsilon', \quad \alpha+\epsilon+\epsilon'<1/4\\
 &ii)\ \gamma^{\ast}T A:L^{\infty,2+\epsilon,0}(D) \rightarrow
 L^{\infty,\epsilon',\delta}(D) \qquad 1/2<\delta<1,
 \quad \epsilon<\epsilon'<1 \\
 &iii)\ A: L^{\infty,\epsilon,\delta}(D)\rightarrow
 L^{\infty,\epsilon',0}(D) \qquad \epsilon<\epsilon',
 \quad \delta<1/2 +(\epsilon'-\epsilon)/2.
\end{align*}
\end{thrm}
\begin{proof}
$i)$. We will prove $i)$
 in the cases that $\lra$,
the kernel of $A$ is of double type $(1,1)$ satisfying the
inequality
\begin{equation*}
|\lra| \lesssim
\frac{\gamma(\zeta)^2}{P^{n-1/2-\mu}|\phi|^{\mu+1}} \quad \mu\ge 1
\end{equation*}
and $\lra$ is of double type $(1,2)$ satisfying
\begin{equation*}
|\lra| \lesssim \frac{\gamma(\zeta)}{P^{n-1-\mu}|\phi|^{\mu+1}}
\quad \mu\ge 1,
\end{equation*}
all other cases being handled by the same methods.
\\ $Case\ a)$.  $\lra$, the kernel of $A$, is of double type
$(1,1)$.

We estimate
\begin{equation}
 \label{hold}
\int_{D}\frac{1}{\gamma^{\epsilon}(\zeta)} \left|
\frac{\gamma(z)^{2-\epsilon'}}{(\phi(\zeta,z))^{\mu+1}P(\zeta,z)^{n-1/2-\mu}}
 -\frac{\gamma(w)^{2-\epsilon'}}{(\phi(\zeta,w))^{\mu+1}P(\zeta,w)^{n-1/2-\mu}}\right|
 dV(\zeta).
\end{equation}
 Then the integral in (\ref{hold})
is bounded by
\begin{align*}
    &\int_{D}\frac{1}{\gamma^{\epsilon}(\zeta)}\left| \frac{
  \gamma(z)^{2-\epsilon'}(\phi(\zeta,w))^{\mu+1}-
\gamma(w)^{2-\epsilon'}(\phi(\zeta,z))^{\mu+1}}{
  (\phi(\zeta,w))^{\mu+1}(\phi(\zeta,z))^{\mu+1}
  P(\zeta,z)^{n-1/2-\mu}}
  \right|dV(\zeta)\\
&\quad+\int_{D}
\frac{\gamma(w)^{2-\epsilon'}}{\gamma^{\epsilon}(\zeta)}\left|
\frac{P(\zeta,z)^{n-1/2-\mu}-P(\zeta,w)^{n-1/2-\mu}}
 {(\phi(\zeta,w))^{\mu+1}P(\zeta,z)^{n-1/2-\mu}P(\zeta,w)^{n-1/2-\mu}
 }
 \right|dV(\zeta)\\
 &\quad= I+II.
\end{align*}

In $I$ we use
\begin{equation*}
(\phi(\zeta,w))^{\mu+1}-(\phi(\zeta,z))^{\mu+1}=
 \sum_{l=0}^{\mu}(\phi(\zeta,w))^{\mu-l}(\phi(\zeta,z))^{l}(\phi(\zeta,w)-\phi(\zeta,z))
\end{equation*}
and
\begin{equation*}
\phi(\zeta,w)-\phi(\zeta,z)
 =O\big(\gamma(\zeta)+|\zeta-z|\big)|z-w|.
\end{equation*}
Therefore
\begin{align*}
I&\lesssim \sum_{l=0}^{\mu}
\int_{D}\frac{\gamma(z)^{2-\epsilon'}}{\gamma^{\epsilon}(\zeta)}
\frac{(\gamma(\zeta)+|\zeta-z|)|z-w|}{|\phi(\zeta,z)|^{\mu+1-l}|\phi(\zeta,w)|^{l+1}
|\zeta-z|^{2n-1-2\mu}}dV(\zeta)\\
 &\qquad+ \int_{D}
\frac{1}{\gamma^{\epsilon}(\zeta)}
\frac{|\gamma(z)^{2-\epsilon'}-\gamma(w)^{2-\epsilon'}|}
 {|\phi(\zeta,w)|^{\mu+1}
|\zeta-z|^{2n-1-2\mu}}dV(\zeta)\\
&\lesssim \sum_{l=0}^{\mu}\int_{D}
\frac{\gamma(z)^{3-\epsilon'}}{\gamma^{\epsilon}(\zeta)}
\frac{|z-w|}{|\phi(\zeta,z)|^{\mu+1-l}|\phi(\zeta,w)|^{l+1}
|\zeta-z|^{2n-1-2\mu}}dV(\zeta)\\
&\qquad+ \sum_{l=0}^{\mu}\int_{D}
 \frac{\gamma(z)^{2-\epsilon'}}{\gamma^{\epsilon}(\zeta)}
\frac{|z-w|}{|\phi(\zeta,z)|^{\mu+1-l}|\phi(\zeta,w)|^{l+1}
|\zeta-z|^{2n-2-2\mu}}dV(\zeta)\\
&\qquad+
 \int_{D} \frac{1}{\gamma^{\epsilon}(\zeta)}
\frac{|\gamma(z)^{2-\epsilon'}-\gamma(w)^{2-\epsilon'}|}
 {|\phi(\zeta,w)|^{\mu+1}
|\zeta-z|^{2n-1-2\mu}}dV(\zeta)\\
 &=I_a+I_b+I_c
\end{align*}
For the integral $I_a$ we break the region of integration into two
parts:
 $\{ |\zeta-w|\le |\zeta-z| \}$ and $\{ |\zeta-z|\le |\zeta-w|
 \}$, and by symmetry we need only consider the region
$\{ |\zeta-z|\le |\zeta-w| \}$.

We first consider the region $E_c$, where $c$ is chosen as in
Lemma \ref{schmlemma}.  Without loss of generality we can choose
$c$ sufficiently small so that $\gamma(z)\lesssim \gamma(\zeta)$
holds in $E_c(z)$.
 We thus estimate
\begin{equation}
\label{intia}
 \int_{D\cap E_c\atop |\zeta-z|\le |\zeta-w|}
\gamma(z)^{3-\epsilon'-\epsilon}\frac{|z-w|}{|\phi(\zeta,z)|^{\mu+1-l}
|\phi(\zeta,w)|^{l+1} |\zeta-z|^{2n-1-2\mu}}dV(\zeta)
 .
\end{equation}
We use $\gamma(z)\lesssim \gamma(w)+|z-w|$ and
\begin{align}
\label{zalpha}
 |z-w|^{\beta}&\lesssim
|\zeta-z|^{\beta}+|\zeta-w|^{\beta}\\
\nonumber &\lesssim
 |\zeta-w|^{\beta}
\end{align}
for $\beta>0$ to bound the integral in (\ref{intia}) by a constant
times
\begin{align}
\nonumber
 |z-w|^{1/2+\alpha}&\int_{D\cap E_c\atop |\zeta-z|\le
|\zeta-w|}\frac{\gamma(z)^2\gamma(w)|\zeta-w|^{1/2-\alpha}}{
 |\phi(\zeta,z)|^{\mu+1-l}
 |\phi(\zeta,w)|^{l+1}|\zeta-z|^{2n-1-2\mu+\epsilon+\epsilon'}}
dV(\zeta)\\
\label{no1}
 \qquad+|z-w|^{\alpha}&\int_{D\cap E_c\atop |\zeta-z|\le |\zeta-w|}
 \frac{\gamma(z)^2 |\zeta-w|^{2-\alpha}}{
 |\phi(\zeta,z)|^{\mu+1-l}
 |\phi(\zeta,w)|^{l+1}|\zeta-z|^{2n-1-2\mu+\epsilon+\epsilon'}}
dV(\zeta).
\end{align}
 We use
  a coordinate system $s_1, s_2,t_1,\ldots,t_{2n-2}$ as
  given by Lemma \ref{schmalzlem} with $s_1=-r(\zeta)$ and $s_2=\mbox{Im} \phi$, and the
estimate (\ref{volest}) on the volume element
\begin{equation}
\label{volestia}
dV(\zeta)\lesssim \frac{t^{2n-3}}{\gamma(z)^2}|ds_1ds_2dt|\\
\end{equation}
where $t=\sqrt{t_1^2+\cdots+t_{2n-2}^2}$, and the second line
follows from $\gamma(\zeta)\lesssim \gamma(z)$ on $E_c(z)$.

We have the estimates
\begin{align*}
 \phi(\zeta,z) & \gtrsim s_1+|s_2| +t^2\\
 \phi(\zeta,w) & \gtrsim -r(w)+s_1+t^2.
\end{align*}

After redefining $s_2$ to be positive, we bound the first integral
of (\ref{no1}) by
\begin{align}
\label{no1est} &
\frac{|z-w|^{1/2+\alpha}}{|r(w)|^{1/2}}\gamma(w)\times
 \\
 \nonumber
 &\qquad\int_V\frac{|\zeta-w|^{1/2-\alpha
 }}
 {(s_1+s_2+t^2)^{\mu+1-l}(s_1+|\zeta-w|^2)^{l+1/2}
  t^{2n-1-2\mu+\epsilon+\epsilon'}}
t^{2n-3}ds_1ds_2dt
\\
\nonumber
&\lesssim\frac{|z-w|^{1/2+\alpha}}{|r(w)|^{1/2}}\gamma(w)\int_V
\frac{t^{2\mu-2-\epsilon-\epsilon'}}{(s_1+s_2+t^2)^{\mu+1-l}(s_1+t^2)^{l+1/4+\alpha/2}}
 ds_1ds_2dt\\
 \nonumber
&\lesssim
 \frac{|z-w|^{1/2+\alpha}}{|r(w)|^{1/2}}\gamma(w)\int_V\frac{1}
 {s_1^{7/8}(s_1+s_2)t^{3/4+\alpha+\epsilon+\epsilon'}}
ds_1ds_2dt\\
\nonumber
 &\lesssim
\frac{|z-w|^{1/2+\alpha}}{|r(w)|^{1/2}}\gamma(w) \int_V\frac{1}
 {s_1^{15/16}s_2^{15/16}t^{3/4+\alpha+\epsilon+\epsilon'}}
ds_1ds_2dt\\
\nonumber
 &\lesssim
  \frac{|z-w|^{1/2+\alpha}}{|r(w)|^{1/2}}\gamma(w),
\end{align}
where $V$ is a bounded subset of $\mathbb{R}^3$.

The second integral of (\ref{no1}) can be bounded by a constant
times
\begin{align*}
|z-w|^{\alpha}
 &\int_V\frac{|\zeta-w|^{2-\alpha
 }}
 {(s_1+s_2+t^2)^{\mu+1-l}(s_1+|\zeta-w|^2)^{l+1}
  t^{2n-1-2\mu+\epsilon+\epsilon'}}
t^{2n-3}ds_1ds_2dt\\
&\lesssim |z-w|^{\alpha}\int_V\frac{t^{2\mu-2-\epsilon-\epsilon'}}
 {(s_1+s_2+t^2)^{\mu+1-l}(s_1+t^2)^{l+\alpha/2}
  }
ds_1ds_2dt\\
&\lesssim |z-w|^{\alpha},
\end{align*}
where again $V$ is a bounded subset of $\mathbb{R}^3$.  The last
line follows by the estimates in (\ref{no1est}).

In estimating the integrals of $I_a$ over the region $D\setminus
E_c$, we write
\begin{align}
\nonumber
 &\int_{D\setminus E_c\atop |\zeta-z|\le
 |\zeta-w|}\frac{1}{\gamma^{\epsilon}(\zeta)}
\frac{|z-w|}{|\phi(\zeta,z)|^{\mu+1-l} |\phi(\zeta,w)|^{l+1}
|\zeta-z|^{2n-4-2\mu+\epsilon'}}dV(\zeta)\\
\nonumber
 &\qquad\lesssim
 |z-w|^{\alpha}\int_{D\setminus E_c\atop |\zeta-z|\le |\zeta-w|}
\frac{1}{\gamma^{\epsilon}(\zeta)}
\frac{|\zeta-w|^{1-\alpha}}{|\phi(\zeta,z)|^{\mu+1-l}
|\phi(\zeta,w)|^{l+1}
|\zeta-z|^{2n-4-2\mu+\epsilon'}}dV(\zeta)\\
\nonumber
 &\qquad\lesssim
 |z-w|^{\alpha}\times\\
 \nonumber
 &\qquad\qquad \int_{D\setminus E_c\atop |\zeta-z|\le |\zeta-w|}
\frac{1}{\gamma^{\epsilon}(\zeta)}
\frac{1}{|\phi(\zeta,z)|^{\mu+1-l}
|\phi(\zeta,w)|^{l+1/2+\alpha/2}
|\zeta-z|^{2n-4-2\mu+\epsilon'}}dV(\zeta)\\
\label{minec} &\qquad\lesssim
 |z-w|^{\alpha}\int_{D\setminus E_c}
 \frac{1}{\gamma^{\epsilon}(\zeta)} \frac{1}{
|\zeta-z|^{2n-1+\alpha+\epsilon'}}dV(\zeta).
\end{align}
We denote the critical points of $r$ by $p_1,\ldots,p_k$, and take
$\varepsilon$ small enough so that in each
\begin{equation*}
U_{2\varepsilon}(p_j)=\{\zeta: D\cap |\zeta-p_j|<2\varepsilon\},
\end{equation*}
for $j=1,\ldots,k$,
 there are coordinates $u_{j_1},\ldots,u_{j_m},v_{j_{m+1}},\ldots,v_{j_{2n}}$ such
 that
 \begin{equation}
 \label{rcoor}
 -r(\zeta)=u_{j_1}^2+\cdots+u_{j_m}^2-v_{j_{m+1}}^2-\cdots - v_{j_{2n}}^2,
 \end{equation}
with $u_{j_{\alpha}}(p_j)=v_{j_{\beta}}(p_j)=0$ for all $1\le
\alpha\le m$ and $m+1\le\beta\le 2n$, from the Morse Lemma.  Let
$U_{\varepsilon}=\bigcup_{j=1}^k U_{\varepsilon}(p_j)$. We break
the problem of estimating (\ref{minec}) into subcases depending on
whether $z\in U_{\varepsilon}$.

  Suppose $z\in U_{\varepsilon}(p_j)$.  Define
$w_1,\ldots,w_{2n}$ by
 \begin{equation}
 \label{defnw}
 w_{\alpha}=\begin{cases}
 u_{j_{\alpha}}\quad \mbox{for } 1\le
\alpha\le m\\
 v_{j_{\alpha}} \quad \mbox{for } m+1\le\alpha\le 2n.
\end{cases}
\end{equation}
Let $x_1,\ldots,x_{2n}$ be defined by
$\zeta_{\alpha}=x_{\alpha}+ix_{n+\alpha}$.  From the Morse Lemma,
the Jacobian of the transformation from coordinates
$x_1,\ldots,x_{2n}$ to $w_1,\ldots,w_{2n}$ is bounded from below
and above and thus we have
\begin{equation*}
|\zeta-z| \simeq |w(\zeta)-w(z)|
\end{equation*}
for $\zeta$, $z\in U_{2\varepsilon}(p_j)$.

 From (\ref{rcoor}) we have
$\gamma(z)\gtrsim |w(z)|$, and thus
\begin{align*}
|w(\zeta)-w(z)| & \simeq |\zeta-z|\\
&
\gtrsim \gamma(z)\\
&\gtrsim |w(z)|\\
&\ge |w(\zeta)| -|w(\zeta)-w(z)|,
\end{align*}
and we obtain
\begin{align*}
|w(\zeta)|&\lesssim |w(\zeta)-w(z)|\\
&\simeq|\zeta-z|.
\end{align*}

Using $|w(\zeta)|\lesssim\gamma(\zeta)$, we estimate, using the
coordinates above
\begin{align*}
|z-w|^{\alpha}\int_{U_{\varepsilon}\setminus E_c}
 \frac{1}{\gamma^{\epsilon}(\zeta)}& \frac{1}{
|\zeta-z|^{2n-1+\alpha+\epsilon'}}dV(\zeta)\\
&\lesssim
 |z-w|^{\alpha}\int_{V}
\frac{u^{m-1}v^{2n-m-1}}{(u+v)^{2n-1+\alpha+\epsilon'+\epsilon}}\\
&\lesssim
 |z-w|^{\alpha},
\end{align*}
where we use $u=\sqrt{u_{j_1}^2+\cdots+u_{j_m}^2}$,
$v=\sqrt{v_{j_{m+1}}^2+\cdots+v_{j_{2n}}^2}$, and $V$ is a bounded
set.

In integrating over the region $D\setminus U_{\varepsilon}$ we
have
\begin{multline*}
|z-w|^{\alpha}\int_{(D\setminus U_{\varepsilon})\setminus E_c}
 \frac{1}{\gamma^{\epsilon}(\zeta)} \frac{1}{
|\zeta-z|^{2n-1+\alpha+\epsilon'}}dV(\zeta)\\ \lesssim
|z-w|^{\alpha}\int_{(D\setminus U_{\varepsilon})\setminus E_c}
 \frac{1}{\gamma^{\epsilon}(\zeta)} dV(\zeta) \lesssim
 |z-w|^{\alpha},
\end{multline*}
which follows by using the coordinates $w_1,\ldots,w_{2n}$ above.

Subcase $b)$.  Suppose $z\notin U_{\varepsilon}$.
 We have $|\zeta-z|\gtrsim \gamma(z)$, but
$\gamma(z)$ is bounded from below, since $z\notin
U_{\varepsilon}$.  We therefore have to estimate
\begin{equation*}
\int_{D}\frac{1}{\gamma^{\epsilon}(\zeta)}dV(\zeta),
\end{equation*}
which is easily done by working with the coordinates
$w_1,\ldots,w_{2n}$ above.

 The region in which $|\zeta-w|\le|\zeta-z|$ is handled in the
same manner, and thus we are finished bounding $I_a$.

We now estimate $I_b$, and again, we only consider the region
$|\zeta-z|\le|\zeta-w|$.  We first estimate the integrals of $I_b$
over the region $E_c(z)$, where $c$ is chosen as in Lemma
\ref{schmalzlem}, and sufficiently small so that
$|\zeta-z|\lesssim\gamma(\zeta)$.   As we chose coordinates for
the integrals in $I_a$, we choose a coordinate system in which
 $s_1=-r(\zeta)$ and $s_2=\mbox{Im} \phi$ and we use the estimate
 on the volume element given by (\ref{volestia}).
We thus write
\begin{align}
 \label{estib}
&\int_{D\cap E_c\atop |\zeta-z|\le|\zeta-w|}\gamma(z)^2
\frac{|z-w|}{|\phi(\zeta,z)|^{\mu+1-l}|\phi(\zeta,w)|^{l+1}
|\zeta-z|^{2n-2-2\mu+\epsilon+\epsilon'}}dV(\zeta)\\
 \nonumber
&\quad\lesssim
 |z-w|^{\alpha}\times\\
 \nonumber
&\qquad\int_{D\cap E_c\atop |\zeta-z|\le|\zeta-w|}
 \gamma(z)^2\frac{1}{|\phi(\zeta,z)|^{\mu+1-l}|\phi(\zeta,w)|^{l+1/2+\alpha/2}
|\zeta-z|^{2n-2-2\mu+\epsilon+\epsilon'}}dV(\zeta)\\
\nonumber
 &\quad\lesssim
 |z-w|^{\alpha}
 \int_V
 \frac{t^{2n-3}}{(s_1+s_2+t^2)^{\mu+1-l}(s_1+t^2)^{l+1/2+\alpha/2}
 t^{2n-2-2\mu+\epsilon+\epsilon'}}ds_1ds_2dt\\
 \nonumber
 &\quad\lesssim
|z-w|^{\alpha}
 \int_0^M\int_0^N
 \frac{t^{2\mu-1-\epsilon-\epsilon'}}{(s_1+t^2)^{\mu-l}(s_1+t^2)^{l+1/2+\alpha/2}
 }ds_1dt\\
 \nonumber
&\quad\lesssim |z-w|^{\alpha}
 \int_0^M\int_0^N
 \frac{1}{s_1^{7/8}t^{1/4+\alpha+\epsilon+\epsilon'}}dsdt\\
\nonumber
 &\quad\lesssim |z-w|^{\alpha},
\end{align}
where we have redefined the coordinate $s_2$ to be positive, $V$
is a bounded subset of $\mathbb{R}^3$, and $M,N>0$ are constants.

The integrals of $I_b$ over the region $D\setminus E_c$ are
estimated by (\ref{minec}) above.

For the integral $I_c$ we use
\begin{equation*}
|\gamma(w)^{2-\epsilon'}-\gamma(z)^{2-\epsilon'}|
 \lesssim |z-w| \left(\gamma(w)^{1-\epsilon'}+\gamma(z)^{1-\epsilon'}\right)
\end{equation*}
and estimate
\begin{equation}
 \label{wandz}
  \int_{D\atop |\zeta-z|\le|\zeta-w|}
  \frac{1}{\gamma^{\epsilon}(\zeta)}
\frac{|z-w|
\left(\gamma(w)^{1-\epsilon'}+\gamma(z)^{1-\epsilon'}\right)}
 {|\phi(\zeta,w)|^{\mu+1}
|\zeta-z|^{2n-1-2\mu}}dV(\zeta).
\end{equation}
Let us first consider the case $\gamma(w)\le \gamma(z)$ and
integrate (\ref{wandz}) over the region $E_c$.  We use a
coordinate system $s, t_1, \ldots, t_{2n-1}$, with $s=-r$ and the
estimate
\begin{equation*}
dV(\zeta) \lesssim \frac{t^{2n-2}}{\gamma(z)}dsdt
\end{equation*}
for $t=\sqrt{t_1^2+\cdots+t_{2n-1}^2}$.  We thus bound
(\ref{wandz}) by
\begin{align}
 \label{3czb}
\int_{D\cap E_c\atop |\zeta-z|\le|\zeta-w|} &\frac{|z-w|
\gamma(z)^{1-\epsilon'}}
 {|\phi(\zeta,w)|^{\mu+1}
|\zeta-z|^{2n-1-2\mu+\epsilon}}dV(\zeta)\\
\nonumber
 &\lesssim |z-w|^{\alpha}
 \int_{D\cap E_c\atop |\zeta-z|\le|\zeta-w|} \frac{
\gamma(z)}
 {|\phi(\zeta,w)|^{\mu+1/2+\alpha/2}
|\zeta-z|^{2n-1-2\mu+\epsilon+\epsilon'}}dV(\zeta)\\
\nonumber
 &\lesssim |z-w|^{\alpha}
 \int_V \frac{t^{2n-2}}{(s+t^2)^{\mu+1/2+\alpha/2}
 t^{2n-1-2\mu+\epsilon+\epsilon'}}dsdt\\
\nonumber
 &\lesssim |z-w|^{\alpha}
 \int_V \frac{1}{s^{3/4}t^{1/2+\epsilon+\epsilon'+\alpha/2}} ds dt\\
 \nonumber
&\lesssim |z-w|^{\alpha},
\end{align}
where $V$ is here a bounded region of $\mathbb{R}^2$.

Over the complement of $E_c$, (\ref{wandz}) is bounded by
\begin{align*}
|z-w|^{\alpha}
 \int_{D\atop |\zeta-z|\le|\zeta-w|} &\frac{1}{\gamma^{\epsilon}(\zeta)}
\frac{1}
 {|\phi(\zeta,w)|^{\mu+1/2+\alpha/2}
|\zeta-z|^{2n-2-2\mu+\epsilon'}}dV(\zeta)\\
&\lesssim |z-w|^{\alpha}
 \int_{D} \frac{1}{\gamma^{\epsilon}(\zeta)}
\frac{1}{|\zeta-z|^{2n-1+\epsilon'+\alpha}}dV(\zeta)\\
&\lesssim |z-w|^{\alpha},
\end{align*}
which follows from the estimates of (\ref{minec}) above.

For the case $\gamma(z)\le\gamma(w)$ we estimate (\ref{wandz})
over the region $E_c$ using coordinates as above by
\begin{align}
 \label{4ae}
&\int_{D\cap E_c\atop
|\zeta-z|\le|\zeta-w|}\frac{1}{\gamma^{\epsilon}(\zeta)}
\frac{|z-w| \gamma(w)^{1-\epsilon'}}
 {|\phi(\zeta,w)|^{\mu+1}
|\zeta-z|^{2n-1-2\mu}}dV(\zeta)\\
\nonumber
 &\qquad \lesssim \frac{|z-w|^{1/2+ \alpha/2}}{|r(w)|^{1/2}}
 \gamma(w)\times\\
\nonumber
 &\qquad\quad \int_{D\cap E_c\atop
|\zeta-z|\le|\zeta-w|} \frac{1}{\gamma^{\epsilon}(\zeta)} \frac{
1}
 {|\phi(\zeta,w)|^{\mu+1/4+\alpha/2}
|\zeta-z|^{2n-1-2\mu+\epsilon'}}dV(\zeta)\\
\nonumber
 &\qquad \lesssim \frac{|z-w|^{1/2+
\alpha/2}}{|r(w)|^{1/2}} \gamma(w)
 \int_V \frac{t^{2n-2}}{(s+t^2)^{\mu+1/4+\alpha/2}
 t^{2n-2\mu+\epsilon+\epsilon'}}dsdt\\
\nonumber
 &\qquad \lesssim \frac{|z-w|^{1/2+
\alpha/2}}{|r(w)|^{1/2}} \gamma(w),
\end{align}
where the last line follows as above.  While over the complement
of $E_c$ we use $\gamma(w)\lesssim |\zeta-w|$ to bound
(\ref{wandz}) by
\begin{align*}
|z-w|^{\alpha}
 \int_{D\atop |\zeta-z|\le|\zeta-w|} &
 \frac{1}{\gamma^{\epsilon}(\zeta)}\frac{1}
 {|\phi(\zeta,w)|^{\mu+\epsilon'/2+\alpha/2}
|\zeta-z|^{2n-1-2\mu}}dV(\zeta)\\
&\lesssim |z-w|^{\alpha}
 \int_{D}\frac{1}{\gamma^{\epsilon}(\zeta)}\frac{1}{|\zeta-z|^{2n-1+\epsilon'+\alpha}}dV(\zeta)\\
&\lesssim |z-w|^{\alpha}.
\end{align*}
We are now done with integral $I$.

 For integral $II$ above we again break the integral into regions
$|\zeta-z|\le|\zeta-w|$ and $|\zeta-w|\le|\zeta-z|$, and we only
consider the region $|\zeta-z|\le|\zeta-w|$, the other case being
handled similarly.

We
 write
\begin{multline*}
\left(P(\zeta,z)^{1/2}\right)^{2n-1-2\mu}
 -\left(P(\zeta,w)^{1/2}\right)^{2n-1-2\mu}
  =\\
  \sum_{l=0}^{2n-2\mu-2} \left(P(\zeta,z)^{1/2}\right)^{2n-2-2\mu-l}
\left(P(\zeta,w)^{1/2}\right)^{l}
\left(P(\zeta,z)^{1/2}-P(\zeta,w)^{1/2}\right),
\end{multline*}
and use
\begin{align*}
\left|P(\zeta,z)^{1/2}-P(\zeta,w)^{1/2}\right|=
 &\frac{|P(\zeta,z)-P(\zeta,w)|}{P(\zeta,z)^{1/2}+P(\zeta,w)^{1/2}}\\
&
 \lesssim
 \frac{|\zeta-z|+\frac{|r(\zeta)|}{\gamma(\zeta)}}{|\zeta-z|}|z-w|\\
&\lesssim
 \frac{|\zeta-w|+\frac{|r(w)|}{\gamma(w)}}{|\zeta-z|}|z-w|
 ,
\end{align*}
which follows from Lemma \ref{c1diff}.

We thus estimate
\begin{align*}
&\int_{D\atop|\zeta-z|\le|\zeta-w|}
  \frac{\gamma(w)^{2-\epsilon'}}{\gamma^{\epsilon}(\zeta)} \left|
\frac{P(\zeta,z)^{n-1/2-\mu}-P(\zeta,w)^{n-1/2-\mu}}
 {(\phi(\zeta,w))^{\mu+1}P(\zeta,z)^{n-1/2-\mu}P(\zeta,w)^{n-1/2-\mu}
 }
 \right|dV(\zeta)\\
 \lesssim&
\sum_{l=0}^{2n-2\mu-2}
 \int_{D\atop|\zeta-z|\le|\zeta-w|}
 \frac{\frac{\gamma(w)^{2-\epsilon'}}{\gamma^{\epsilon}(\zeta)}|z-w|
 \left(|\zeta-z|+\frac{|r(w)|}{\gamma(w)}\right)dV(\zeta)}
 {|\phi(\zeta,w)|^{\mu+1}\left(P(\zeta,z)^{1/2}\right)^{l+1}
 \left(P(\zeta,w)^{1/2}\right)^{2n-1-2\mu-l}|\zeta-z|}\\
\lesssim&
 \int_{D\atop|\zeta-z|\le|\zeta-w|}
\frac{\gamma(w)^{2-\epsilon'}}{\gamma^{\epsilon}(\zeta)}
 \frac{|z-w|}
 {|\phi(\zeta,w)|^{\mu+1}|\zeta-z|^{2n-2\mu}}dV(\zeta)\\
 &+\int_{D\atop|\zeta-z|\le|\zeta-w|}
\frac{\gamma(w)^{1-\epsilon'}}{\gamma^{\epsilon}(\zeta)}
 \frac{|r(w)||z-w|}
 {|\phi(\zeta,w)|^{\mu+1}|\zeta-z|^{2n+1-2\mu}}dV(\zeta)\\
 =&II_a+II_b.
\end{align*}

For $II_a$, we break the integral into the regions $E_c(z)$ and
its complement. We first consider
\begin{align}
\label{ivnec} \int_{D\setminus E_c\atop|\zeta-z|\le|\zeta-w|}&
\frac{\gamma(w)^{2-\epsilon'}}{\gamma^{\epsilon}(\zeta)}
 \frac{|z-w|}
 {|\phi(\zeta,w)|^{\mu+1}|\zeta-z|^{2n-2\mu}}dV(\zeta)\\
 \nonumber
& \lesssim |z-w|^{\alpha} \int_{D\setminus
E_c\atop|\zeta-z|\le|\zeta-w|}\frac{1}{\gamma^{\epsilon}(\zeta)}
 \frac{1}
 {|\phi(\zeta,w)|^{\mu-1/2+\alpha/2+\epsilon'/2}|\zeta-z|^{2n-2\mu}}dV(\zeta)\\
 \nonumber
&\lesssim |z-w|^{\alpha}
\int_{D}\frac{1}{\gamma^{\epsilon}(\zeta)}
 \frac{1}{|\zeta-z|^{2n-1+\alpha+\epsilon'}} dV(\zeta)\\
 \nonumber
&\lesssim |z-w|^{\alpha},
\end{align}
where we use $\gamma(w)\lesssim |\zeta-w|$ and the estimates for
(\ref{minec}).

We then bound the integral $II_a$ over the region $E_c(z)$ by
considering the different cases $\gamma(w)\le\gamma(z)$ and
$\gamma(z)\le\gamma(w)$.  In the case $\gamma(w)\le\gamma(z)$, we
use a coordinate system, $s, t_1,\ldots, t_{2n-1}$, in which
$s=-r(\zeta)$, and using the estimate
\begin{equation}
 \label{vest1}
dV(\zeta)\lesssim
 \frac{t^{2n-2}}{\gamma(z)}dsdt
,
\end{equation} we
  have
\begin{align}
 \label{3cwe}
&\int_{D\cap E_c\atop|\zeta-z|\le|\zeta-w|}
 \frac{\gamma(w)^{2-\epsilon'}}{\gamma^{\epsilon}(\zeta)}
 \frac{|z-w|}
 {|\phi(\zeta,w)|^{\mu+1}|\zeta-z|^{2n-2\mu}}dV(\zeta)\\
\nonumber
 &\qquad\lesssim
 \frac{|z-w|^{1/2+\alpha}}{|r(w)|^{1/2}}\gamma(w)
 \int_{D\cap E_c\atop|\zeta-z|\le|\zeta-w|}\gamma(z)^{1-\epsilon'}
 \frac{dV(\zeta)}
 {|\phi(\zeta,w)|^{\mu+1/4+\alpha/2}|\zeta-z|^{2n-2\mu+\epsilon}}\\
\nonumber
 &\qquad\lesssim
 \frac{|z-w|^{1/2+\alpha}}{|r(w)|^{1/2}}\gamma(w) \int_{V}
 \frac{t^{2\mu-2-\epsilon-\epsilon'}}
 {(s+t^2)^{\mu+1/4+\alpha/2}}dsdt\\
\nonumber
 &\qquad\lesssim
 \frac{|z-w|^{1/2+\alpha}}{|r(w)|^{1/2}}\gamma(w) \int_{V}
 \frac{1}
 {s^{7/8}t^{3/4+\alpha+\epsilon+\epsilon'}}dsdt\\
\nonumber
 &\qquad\lesssim
 \frac{|z-w|^{1/2+\alpha}}{|r(w)|^{1/2}}\gamma(w).
\end{align}
In the case $\gamma(z)\le\gamma(w)$, we estimate as above
\begin{align*}
&\int_{D\cap E_c\atop|\zeta-z|\le|\zeta-w|}
\frac{\gamma(w)^{2-\epsilon'}}{\gamma^{\epsilon}(\zeta)}
 \frac{|z-w|}
 {|\phi(\zeta,w)|^{\mu+1}|\zeta-z|^{2n-2\mu}}dV(\zeta)\\
&\qquad\lesssim
 \frac{|z-w|^{1/2+\alpha}}{|r(w)|^{1/2}}\gamma(w)
 \int_{D\cap E_c\atop|\zeta-z|\le|\zeta-w|}\gamma(w)
 \frac{dV(\zeta)}
 {|\phi(\zeta,w)|^{\mu+1/4+\alpha/2}|\zeta-z|^{2n-2\mu+\epsilon+\epsilon'}}\\
&\qquad\lesssim
 \frac{|z-w|^{1/2+\alpha}}{|r(w)|^{1/2}}\gamma(w)
 \int_{D\cap E_c\atop|\zeta-z|\le|\zeta-w|}
 \frac{(\gamma(z)+|\zeta-w|)dV(\zeta)}
 {|\phi(\zeta,w)|^{\mu+1/4+\alpha/2}|\zeta-z|^{2n-2\mu+\epsilon+\epsilon'}}.
\end{align*}
The integral involving $\gamma(z)$ is estimated exactly as above.
We thus have to deal with
\begin{equation*}
\frac{|z-w|^{1/2+\alpha}}{|r(w)|^{1/2}}\gamma(w)
 \int_{D\cap E_c\atop|\zeta-z|\le|\zeta-w|}
 \frac{|\zeta-w|dV(\zeta)}
 {|\phi(\zeta,w)|^{\mu+1/4+\alpha/2}|\zeta-z|^{2n-2\mu+\epsilon+\epsilon'}},
\end{equation*}
which we estimate using the coordinates $s,t_1,\ldots,t_{2n-1}$
above by
\begin{align*}
\frac{|z-w|^{1/2+\alpha}}{|r(w)|^{1/2}}&\gamma(w)
 \int_{D\cap E_c\atop|\zeta-z|\le|\zeta-w|}
 \frac{|\zeta-w|dV(\zeta)}
 {|\phi(\zeta,w)|^{\mu+1/4+\alpha/2}|\zeta-z|^{2n-2\mu+\epsilon+\epsilon'}}\\
&\lesssim \frac{|z-w|^{1/2+\alpha}}{|r(w)|^{1/2}}\gamma(w)
 \int_{V}
 \frac{t^{2n-2}dsdt}
 {(s+t^2)^{\mu-1/4+\alpha/2}(s+t)^{2n-2\mu+1+\epsilon+\epsilon'}}\\
&\lesssim \frac{|z-w|^{1/2+\alpha}}{|r(w)|^{1/2}}\gamma(w)
 \int_V \frac{1}{s^{3/4+\alpha/2+\epsilon+\epsilon'+\delta}t^{1-\delta}}
 dsdt\\
&\lesssim \frac{|z-w|^{1/2+\alpha}}{|r(w)|^{1/2}}\gamma(w),
\end{align*}
where $0<\delta<1/4-(\alpha/2+\epsilon+\epsilon')$.

For $II_b$ we first estimate
\begin{align*}
\int_{D\setminus
E_c\atop|\zeta-z|\le|\zeta-w|}&\frac{\gamma(w)^{1-\epsilon'}}{\gamma^{\epsilon}(\zeta)}
 \frac{|r(\zeta)||z-w|}
 {|\phi(\zeta,w)|^{\mu+1}|\zeta-z|^{2n+1-2\mu}}dV(\zeta)\\
&\lesssim |z-w|^{\alpha}
 \int_{D\setminus E_c\atop|\zeta-z|\le|\zeta-w|}
 \frac{1}{\gamma^{\epsilon}(\zeta)}\frac{|\zeta-w|^{2-\alpha-\epsilon'}}
 {|\phi(\zeta,w)|^{\mu}|\zeta-z|^{2n+1-2\mu}}dV(\zeta)\\
&\lesssim |z-w|^{\alpha}
 \int_{D}
 \frac{1}{\gamma^{\epsilon}(\zeta)}\frac{1}
 {|\zeta-z|^{2n-1+\alpha+\epsilon'}}dV(\zeta)\\
&\lesssim |z-w|^{\alpha},
\end{align*}
where $c$ is chosen as in Lemma \ref{schmalzlem} and we use
$\gamma(w)\lesssim |\zeta-w|$ on $D\setminus E_c(z)$.

We now finish the estimates for $II_b$.  We have
\begin{align}
 \nonumber
&\int_{D\cap E_c\atop|\zeta-z|\le|\zeta-w|}
 \frac{\gamma(w)^{1-\epsilon'}}{\gamma^{\epsilon}(\zeta)}
 \frac{|r(\zeta)||z-w|}
 {|\phi(\zeta,w)|^{\mu+1}|\zeta-z|^{2n+1-2\mu}}dV(\zeta)\\
\label{2bsep}
 &\qquad\lesssim |z-w|^{\alpha}
 \int_{D\cap E_c\atop|\zeta-z|\le|\zeta-w|}
\gamma(w)^{1-\epsilon'}
 \frac{1}
 {|\phi(\zeta,w)|^{\mu-1/2+\alpha/2}|\zeta-z|^{2n+1-2\mu+\epsilon}}dV(\zeta).
\end{align}
We again consider the different cases $\gamma(w)\le \gamma(z)$ and
$\gamma(z)\le \gamma(w)$ separately.
 With $\gamma(w)\le\gamma(z)$, we use coordinates $s,t_1,\ldots,t_{2n-1}$
as above with the volume estimate (\ref{vest1}) to estimate
(\ref{2bsep}) by
\begin{align*}
 |z-w|^{\alpha}
 \int_{V}&
 \frac{t^{2n-2}}
 {(s+t^2)^{\mu-1/2+\alpha/2}(s+t)^{2n+1-2\mu+\epsilon+\epsilon'}}dsdt\\
  &\lesssim |z-w|^{\alpha}
 \int_{V}
 \frac{1}
 {s^{1/2+\alpha/2+\epsilon+\epsilon'+\delta}t^{1-\delta}}dsdt\\
 &\lesssim |z-w|^{\alpha},
\end{align*}
where $0<\delta<1/2-(\alpha/2+\epsilon+\epsilon')$, and $V$ again
denotes a bounded subset of $\mathbb{R}^2$.

In the case $\gamma(z)\le\gamma(w)$, we write $\gamma(w)\lesssim
\gamma(z)+|\zeta-w|$, and estimate (\ref{2bsep}) by
\begin{equation*}
|z-w|^{\alpha}
 \int_{D\cap E_c\atop|\zeta-z|\le|\zeta-w|}
 \frac{\gamma(z)+|\zeta-w|}
 {|\phi(\zeta,w)|^{\mu-1/2+\alpha/2}|\zeta-z|^{2n+1-2\mu+\epsilon+\epsilon'}}
 dV(\zeta).
\end{equation*}
The integral involving $\gamma(z)$ is handled exactly as above, so
we estimate
\begin{align*}
|z-w|^{\alpha}
 &\int_{D\cap E_c\atop|\zeta-z|\le|\zeta-w|}
 \frac{|\zeta-w|}
 {|\phi(\zeta,w)|^{\mu-1/2+\alpha/2}|\zeta-z|^{2n+1-2\mu+\epsilon+\epsilon'}}
 dV(\zeta)\\
&\lesssim |z-w|^{\alpha}
 \int_{D\cap E_c\atop|\zeta-z|\le|\zeta-w|}
\frac{1}
 {|\phi(\zeta,w)|^{\mu-1+\alpha/2}|\zeta-z|^{2n+1-2\mu+\epsilon+\epsilon'}}
 dV(\zeta).
\end{align*}
The case of $\mu=1$ is trivial so we assume $\mu\ge2$ and using
the coordinates $s, t_1,\ldots,t_{2n-1}$, we estimate
\begin{align*}
|z-w|^{\alpha}
 \int_{V}&
 \frac{t^{2n-2}}
 {(s+t^2)^{\mu-1+\alpha/2}(s+t)^{2n+2-2\mu+\epsilon+\epsilon'}}dsdt\\
&\lesssim |z-w|^{\alpha}
 \int_{V}
 \frac{1}{s^{3/4+\alpha/2+\epsilon+\epsilon'}t^{1/2}} ds dt\\
&\lesssim |z-w|^{\alpha}.
\end{align*}
\\
$Case\ b)$.  $\lra$ is of double type $(1,2)$.

Following the arguments above we see we need to estimate
\begin{align*}
    &\int_{D}\frac{1}{\gamma(\zeta)^{1+\epsilon}}\left| \frac{
  \gamma(z)^{2-\epsilon'}(\phi(\zeta,w))^{\mu+1}-
 \gamma(w)^{2-\epsilon'}(\phi(\zeta,z))^{\mu+1}}{
  (\phi(\zeta,w))^{\mu+1}(\phi(\zeta,z))^{\mu+1}
  P(\zeta,z)^{n-1-\mu}}
  \right|dV(\zeta)\\
&\quad+\int_{D}
\frac{\gamma(w)^{2-\epsilon'}}{\gamma(\zeta)^{1+\epsilon}}\left|
\frac{P(\zeta,z)^{n-1-\mu}-P(\zeta,w)^{n-1-\mu}}
 {(\phi(\zeta,w))^{\mu+1}P(\zeta,z)^{n-1-\mu}P(\zeta,w)^{n-1-\mu}
 }
 \right|dV(\zeta)\\
 &\quad= III+IV.
\end{align*}

Following the calculations for integral $I$ in case $a)$ we
estimate integral $III$ by the integrals
\begin{align*}
\sum_{l=0}^{\mu}\int_{D}\frac{\gamma(z)^{2-\epsilon'}}{\gamma(\zeta)^{\epsilon}}&
\frac{|z-w|}{|\phi(\zeta,z)|^{\mu+1-l}|\phi(\zeta,w)|^{l+1}
|\zeta-z|^{2n-2-2\mu}}dV(\zeta)\\
&\quad+
\sum_{l=0}^{\mu}\int_{D}\frac{\gamma(z)^{2-\epsilon'}}{\gamma(\zeta)^{1+\epsilon}}
\frac{|z-w|}{|\phi(\zeta,z)|^{\mu+1-l}|\phi(\zeta,w)|^{l+1}
|\zeta-z|^{2n-3-2\mu}}dV(\zeta)\\
   &\quad+
 \int_{D}\frac{1}{\gamma(\zeta)^{1+\epsilon}}
\frac{|\gamma(z)^{2-\epsilon'}-\gamma(w)^{2-\epsilon'}|}
 {|\phi(\zeta,w)|^{\mu+1}
|\zeta-z|^{2n-2-2\mu}}dV(\zeta)\\
 &=III_a+III_b+III_c.
\end{align*}
Estimates for the integral $III_a$ are given by $I_b$ in case
$a)$.

For the integrals of $III_b$, we consider separately the regions
$E_c(z)$ and its complement.  We also only consider the case
$|\zeta-z|\le|\zeta-w|$.

In the region $D\cap E_c(z)$, we use a coordinate system in which
$s=-r(\zeta)$ is a coordinate, and we use the estimate on the
volume element in $E_c(z)$ given by (\ref{vest1}).  We can also
assume that $c$ is sufficiently small to guarantee that
$|\zeta-z|\lesssim\gamma(\zeta)$ in $E_c$.

The integrals
\begin{equation*}
\int_{D\cap E_c\atop |\zeta-z|\le|\zeta-w|}
\frac{\gamma(z)^{2-\epsilon'}}{\gamma(\zeta)^{1+\epsilon}}
\frac{|z-w|}{|\phi(\zeta,z)|^{\mu+1-l}|\phi(\zeta,w)|^{l+1}
|\zeta-z|^{2n-3-2\mu}}dV(\zeta)
\end{equation*}
can thus be bounded by
\begin{align*}
\nonumber
&\frac{|z-w|^{1/2+\alpha}}{|r(z)|^{1/2}}\gamma(z)\times\\
&\int_{V}\frac{|\zeta-w|^{1/2-\alpha}}
 {(s+|\zeta-z|^2)^{\mu+1/2-l}(s+|\zeta-w|^2)^{l+1}|\zeta-z|^{2n-2-2\mu+\epsilon+\epsilon'}}
 t^{2n-2}dsdt\\
\nonumber
 &\lesssim\frac{|z-w|^{1/2+\alpha}}{|r(z)|^{1/2}}\gamma(z)
 \int_{V}\frac{t^{2n-2}}
 {(s+|\zeta-z|^2)^{\mu+5/4+\alpha/2}|\zeta-z|^{2n-2-2\mu+\epsilon+\epsilon'}}
 dsdt\\
 \nonumber
&\lesssim\frac{|z-w|^{1/2+\alpha}}{|r(z)|^{1/2}}\gamma(z)
 \int_V\frac{t^{2\mu-\epsilon-\epsilon'}}{(s+t^2)^{\mu+5/4+\alpha/2}}dsdt\\
 \nonumber
&\lesssim\frac{|z-w|^{1/2+\alpha}}{|r(z)|^{1/2}}\gamma(z)
 \int_V\frac{1}{s^{7/8}t^{3/4+\alpha+\epsilon+\epsilon'}}dsdt\\
\nonumber
 &\lesssim\frac{|z-w|^{1/2+\alpha}}{|r(z)|^{1/2}}\gamma(z)
,
\end{align*}
where $V$ is a bounded subset of $\mathbb{R}^2$.

We now estimate
\begin{align*}
\int_{D\setminus E_c\atop
|\zeta-z|\le|\zeta-w|}&\frac{\gamma(z)^{2-\epsilon'}}{\gamma(\zeta)^{1+\epsilon}}
\frac{|z-w|}{|\phi(\zeta,z)|^{\mu+1-l}|\phi(\zeta,w)|^{l+1}
|\zeta-z|^{2n-3-2\mu}}dV(\zeta)\\
&\lesssim
 |z-w|^{\alpha}\int_{D\setminus E_c\atop |\zeta-z|\le|\zeta-w|}
\frac{1}{\gamma(\zeta)^{1+\epsilon}}\frac{1}{|\phi(\zeta,w)|^{l+1/2+\alpha/2}
|\zeta-z|^{2n-3-2l+\epsilon'}}dV(\zeta).
\end{align*}
We use coordinates
 $u_{j_1},\ldots,u_{j_m},v_{j_{m+1}},\ldots,v_{j_{2n}}$ as in
 (\ref{rcoor}) and the neighborhoods $U_{2\varepsilon}(p_j)$
 defined above.
We break the problem into subcases depending on whether $z\in
U_{\varepsilon}$.

 Subcase $a)$.  Suppose $z\in U_{\varepsilon}(p_j)$.  As we did above above
 define $w_1,\ldots,w_{2n}$
by
 \begin{equation*}
 w_{\alpha}=\begin{cases}
 u_{j_{\alpha}}\quad \mbox{for } 1\le
\alpha\le m\\
 v_{j_{\alpha}} \quad \mbox{for } m+1\le\alpha\le 2n,
\end{cases}
\end{equation*}
and let $x_1,\ldots,x_{2n}$ be defined by
$\zeta_{\alpha}=x_{\alpha}+ix_{n+\alpha}$.
  Recall that we have
$ |w(\zeta)|\lesssim |\zeta-z| $ and
$|w(\zeta)|\lesssim\gamma(\zeta)$.  Thus we estimate,
 using the
coordinates above,
\begin{align}
 \label{3cznote}
 |z-w|^{\alpha}\int_{D\setminus E_c\atop |\zeta-z|\le|\zeta-w|}&
\frac{1}{\gamma(\zeta)^{1+\epsilon}}\frac{1}{|\phi(\zeta,w)|^{l+1/2+\alpha/2}
|\zeta-z|^{2n-3-2l+\epsilon'}}dV(\zeta)\\
\nonumber
 &\lesssim |z-w|^{\alpha}
 \int_{D\setminus E_c\atop |\zeta-z|\le|\zeta-w|}
 \frac{1}{\gamma(\zeta)^{1+\epsilon}}
 \frac{1}{|\zeta-z|^{2n-2+\alpha+\epsilon'}}dV(\zeta)\\
\nonumber
 &\lesssim |z-w|^{\alpha}
 \int_V \frac{u^{m-1}v^{2n-m-1}}{(u+v)^{2n-1+\alpha+\epsilon+\epsilon'}}
 dudv\\
\nonumber
 &\lesssim |z-w|^{\alpha}
 \int_V \frac{1}{u^{1/2}v^{1/2+\alpha+\epsilon+\epsilon'}}
 dudv\\
\nonumber
 &\lesssim |z-w|^{\alpha},
\end{align}
where we use $u=\sqrt{u_{j_1}^2+\cdots+u_{j_m}^2}$,
$v=\sqrt{v_{j_{m+1}}^2+\cdots+v_{j_{2n}}^2}$, and $V$ is a bounded
set.

 Subcase $b)$.  Suppose $z\notin U_{\varepsilon}$.
 We have $|\zeta-z|\gtrsim \gamma(z)$, but
$\gamma(z)$ is bounded from below, since $z\notin
U_{\varepsilon}$.  We therefore have to estimate
\begin{equation*}
\int_{D}\frac{1}{\gamma(\zeta)^{1+\epsilon}}dV(\zeta),
\end{equation*}
which is easily done by working with the coordinates
$w_1,\ldots,w_{2n}$ above.

 We now estimate integral $III_c$.  We
use
\begin{equation*}
|\gamma(z)^{2-\epsilon'}-\gamma(w)^{2-\epsilon'}|
 \lesssim |z-w|\left(
 \gamma(z)^{1-\epsilon'}+\gamma(w)^{1-\epsilon'}\right)
\end{equation*}
 to write
\begin{equation*}
 \label{3c}
 III_c\lesssim
 \int_{D\atop |\zeta-z|\le|\zeta-w|}
\frac{\gamma(z)^{1-\epsilon'}+\gamma(w)^{1-\epsilon'}}
 {\gamma(\zeta)^{1+\epsilon}}
\frac{|z-w|}{|\phi(\zeta,w)|^{\mu+1}|\zeta-z|^{2n-2-2\mu}}dV(\zeta).
\end{equation*}
We first assume $\gamma(w)\le\gamma(z)$.  Then we estimate
\begin{equation}
 \label{zbig}
 \int_{D\atop |\zeta-z|\le|\zeta-w|}
\frac{\gamma(z)^{1-\epsilon'}}
 {\gamma(\zeta)^{1+\epsilon}}
\frac{|z-w|}{|\phi(\zeta,w)|^{\mu+1}|\zeta-z|^{2n-2-2\mu}}dV(\zeta).
\end{equation}
 by breaking the integral into the regions $E_c$ and
$D\setminus E_c$.  In $E_c$, again assuming $c$ is sufficiently
small so that $|\zeta-z|\lesssim \gamma(\zeta)$, (\ref{zbig}) is
bounded by
\begin{equation*}
 \int_{D\atop |\zeta-z|\le|\zeta-w|}
\gamma(z)^{1-\epsilon'}
\frac{|z-w|}{|\phi(\zeta,w)|^{\mu+1}|\zeta-z|^{2n-1-2\mu+\epsilon}}dV(\zeta),
\end{equation*}
which we showed to be bounded by $|z-w|^{\alpha}$ in (\ref{3czb}).
 In the region $D\setminus E_c$, we estimate
\begin{align*}
  \int_{D\setminus E_c\atop |\zeta-z|\le|\zeta-w|}&
\frac{\gamma(z)^{1-\epsilon'}}
 {\gamma(\zeta)^{1+\epsilon}}
\frac{|z-w|}{|\phi(\zeta,w)|^{\mu+1}|\zeta-z|^{2n-2-2\mu}}dV(\zeta)\\
 &\lesssim
|z-w|^{\alpha}\int_{D\setminus E_c\atop |\zeta-z|\le|\zeta-w|}
\frac{1}
 {\gamma(\zeta)^{1+\epsilon}}
\frac{1}{|\phi(\zeta,w)|^{\mu+1/2+\alpha/2}|\zeta-z|^{2n-3-2\mu+\epsilon'}}
 dV(\zeta)\\
&\lesssim |z-w|^{\alpha},
\end{align*}
where the last line follows from (\ref{3cznote}) above.

We therefore now consider the case $\gamma(z)\le\gamma(w)$ so that
\begin{equation*}
III_c\lesssim
 \int_{D\atop |\zeta-z|\le|\zeta-w|}
\frac{\gamma(w)^{1-\epsilon'}}
 {\gamma(\zeta)^{1+\epsilon}}
\frac{|z-w|}{|\phi(\zeta,w)|^{\mu+1}|\zeta-z|^{2n-2-2\mu}}dV(\zeta).
\end{equation*}
In the region $E_c$ we estimate
\begin{align*}
&\int_{D\cap E_c\atop |\zeta-z|\le|\zeta-w|}
\frac{\gamma(w)^{1-\epsilon'}}
 {\gamma(\zeta)^{1+\epsilon}}
\frac{|z-w|}{|\phi(\zeta,w)|^{\mu+1}|\zeta-z|^{2n-2-2\mu}}dV(\zeta)\\
&\qquad\lesssim
 \int_{D\cap E_c\atop |\zeta-z|\le|\zeta-w|}
\gamma(w)
\frac{|z-w|}{|\phi(\zeta,w)|^{\mu+1}|\zeta-z|^{2n-1-2\mu+\epsilon+\epsilon'}}dV(\zeta)\\
&\qquad\lesssim
 \frac{|z-w|^{1/2+\alpha}}{|r(w)|^{1/2}}\gamma(w)
\int_{D\cap E_c\atop |\zeta-z|\le|\zeta-w|}
 \frac{1}{|\phi(\zeta,w)|^{\mu+1/4+\alpha/2}
 |\zeta-z|^{2n-1-2\mu+\epsilon+\epsilon'}}dV(\zeta).
\end{align*}
Using the coordinate system $s=-r(\zeta)$, $t_1\ldots,t_{2n-2}$
with volume estimate (\ref{vest1}) as above we can estimate
\begin{equation*}
\frac{|z-w|^{1/2+\alpha}}{|r(w)|^{1/2}}\gamma(w)
 \int_V \frac{t^{2n-2}}{(s+t^2)^{\mu+1/4+\alpha/2}
 t^{2n-2\mu+\epsilon+\epsilon'}} dsdt \lesssim
\frac{|z-w|^{1/2+\alpha}}{|r(w)|^{1/2}}\gamma(w)
\end{equation*}
by (\ref{3cwe}).

In the region $D\setminus E_c$, we use $\gamma(w)\lesssim
|\zeta-w|$ to estimate
\begin{align}
 \label{4anote}
\int_{D\setminus E_c\atop |\zeta-z|\le|\zeta-w|}&
\frac{\gamma(w)^{1-\epsilon'}}
 {\gamma(\zeta)^{1+\epsilon}}
\frac{|z-w|}{|\phi(\zeta,w)|^{\mu+1}|\zeta-z|^{2n-2-2\mu}}dV(\zeta)\\
\nonumber
 &\lesssim
 \int_{D\setminus E_c\atop |\zeta-z|\le|\zeta-w|}
\frac{1}{\gamma(\zeta)^{1+\epsilon}}
 \frac{|\zeta-w|^{1-\epsilon'}|z-w|}
 {|\phi(\zeta,w)|^{\mu+1}|\zeta-z|^{2n-2-2\mu}}dV(\zeta)\\
\nonumber
 &\lesssim
 |z-w|^{\alpha}
 \int_{D\setminus E_c\atop
|\zeta-z|\le|\zeta-w|} \frac{1}{\gamma(\zeta)^{1+\epsilon}}
 \frac{1}
 {|\phi(\zeta,w)|^{\mu+\epsilon'/2+\alpha/2}
 |\zeta-z|^{2n-2-2\mu}}dV(\zeta)\\
\nonumber
 &\lesssim
 |z-w|^{\alpha}
 \int_{D\setminus E_c\atop
|\zeta-z|\le|\zeta-w|} \frac{1}{\gamma(\zeta)^{1+\epsilon}}
 \frac{1}
 {
 |\zeta-z|^{2n-2+\epsilon'+\alpha}}dV(\zeta)\\
\nonumber
 &\lesssim
 |z-w|^{\alpha}
 \int_V \frac{u^{m-1}v^{2n-1-m}}{(u+v)^{2n-1+\epsilon+\epsilon'+\alpha}}
dudv\\
\nonumber &\lesssim
 |z-w|^{\alpha},
\end{align}
where the coordinates $u$ and $v$ are defined as in
(\ref{3cznote}), and where the last line follows from
(\ref{3cznote}).  We are now done estimating integral $III$ and we
turn to $IV$.

 As in case $a)$ for
integral $II$ we estimate $IV$ by the integrals
\begin{align*}
&
 \int_{D\atop|\zeta-z|\le|\zeta-w|}
\frac{\gamma(w)^{2-\epsilon'}}{\gamma(\zeta)^{1+\epsilon}}
 \frac{|z-w|}
 {|\phi(\zeta,w)|^{\mu+1}|\zeta-z|^{2n-1-2\mu}}dV(\zeta)\\
 &\qquad+\int_{D\atop|\zeta-z|\le|\zeta-w|}
 \frac{\gamma(w)^{2-\epsilon'}}{\gamma(\zeta)^{2+\epsilon}}
 \frac{|r(\zeta)||z-w|}
 {|\phi(\zeta,w)|^{\mu+1}|\zeta-z|^{2n-2\mu}}dV(\zeta)\\
 =&IV_a+IV_b.
\end{align*}

To estimate $IV_a$ we break the region of integration in $E_c$ and
$D\setminus E_c$.  In the region $D\setminus E_c$ we use
$\gamma(w)\lesssim |\zeta-w|$ and estimate
\begin{align*}
&\int_{D\setminus E_c\atop|\zeta-z|\le|\zeta-w|}
\frac{1}{\gamma(\zeta)^{1+\epsilon}}
 \frac{|z-w|}
 {|\phi(\zeta,w)|^{\mu+\epsilon'/2}|\zeta-z|^{2n-1-2\mu}}dV(\zeta)\\
&\qquad \lesssim |z-w|^{\alpha}
 \int_{D\setminus E_c\atop|\zeta-z|\le|\zeta-w|}
 \frac{1}{\gamma(\zeta)^{1+\epsilon}}
 \frac{1}
 {|\phi(\zeta,w)|^{\mu+\epsilon'/2-1/2+\alpha/2}
|\zeta-z|^{2n-1-2\mu}}dV(\zeta)\\
&\qquad \lesssim |z-w|^{\alpha}
 \int_{D\setminus E_c\atop|\zeta-z|\le|\zeta-w|}
 \frac{1}{\gamma(\zeta)^{1+\epsilon}}
 \frac{1}
 {|\zeta-z|^{2n-2-2\mu+\epsilon'+\alpha}}dV(\zeta)\\
&\qquad \lesssim |z-w|^{\alpha},
\end{align*}
where the last line follows from (\ref{4anote}).

In the region $E_c$ we consider the different cases
$\gamma(w)\le\gamma(z)$ and $\gamma(z)\le\gamma(w)$ separately. In
the case $\gamma(w)\le\gamma(z)$, we write
\begin{align*}
&\int_{D\cap E_c\atop|\zeta-z|\le|\zeta-w|}
\frac{\gamma(w)^{2-\epsilon'}}{\gamma(\zeta)^{1+\epsilon}}
 \frac{|z-w|}
 {|\phi(\zeta,w)|^{\mu+1}|\zeta-z|^{2n-1-2\mu}}dV(\zeta)\\
&\qquad \lesssim \gamma(w)\int_{D\cap
E_c\atop|\zeta-z|\le|\zeta-w|}
\frac{\gamma(z)}{\gamma(\zeta)^{1+\epsilon}}
 \frac{|z-w|}
 {|\phi(\zeta,w)|^{\mu+1}|\zeta-z|^{2n-1-2\mu+\epsilon'}}dV(\zeta)\\
&\qquad \lesssim \frac{|z-w|^{1/2+\alpha}}{|r(w)|^{1/2}}\gamma(w)
 \int_{D\cap
E_c\atop|\zeta-z|\le|\zeta-w|} \gamma(z)
 \frac{1}
 {|\phi(\zeta,w)|^{\mu+1/4+\alpha/2}
 |\zeta-z|^{2n-2\mu+\epsilon+\epsilon'}}dV(\zeta)
\end{align*}
and we choose a coordinate system in which
 $s=-r(\zeta)$ and we use the estimate
 on the volume element given by (\ref{vest1})
to reduce the estimate to
\begin{equation*}
 \frac{|z-w|^{1/2+\alpha}}{|r(w)|^{1/2}}\gamma(w)
 \int_{V}
 \frac{t^{2\mu-2-\epsilon-\epsilon'}}
 {(s+t^2)^{\mu+1/4+\alpha/2}}dsdt \lesssim
 \frac{|z-w|^{1/2+\alpha}}{|r(w)|^{1/2}}\gamma(w)
\end{equation*}
which follows from (\ref{3cwe}).

In the case $\gamma(z)\le \gamma(w)$ we have
\begin{multline*}
\int_{D\cap E_c\atop|\zeta-z|\le|\zeta-w|}
\frac{\gamma(w)^{2-\epsilon'}}{\gamma(\zeta)^{1+\epsilon}}
 \frac{|z-w|}
 {|\phi(\zeta,w)|^{\mu+1}|\zeta-z|^{2n-1-2\mu}}dV(\zeta)\\
\lesssim
 \frac{|z-w|^{1/2+\alpha}}{|r(w)|^{1/2}}\gamma(w)
\int_{D\cap E_c\atop|\zeta-z|\le|\zeta-w|} \gamma(w)
 \frac{1}
 {|\phi(\zeta,w)|^{\mu+1/4+\alpha/2}|\zeta-z|^{2n-2\mu+\epsilon+\epsilon'}}dV(\zeta).
\end{multline*}
We then write $\gamma(w)\lesssim \gamma(z)+|\zeta-w|$, and we
bound
\begin{multline*}
\frac{|z-w|^{1/2+\alpha}}{|r(w)|^{1/2}}\gamma(w) \int_{D\cap
E_c\atop|\zeta-z|\le|\zeta-w|} \gamma(z)
 \frac{1}
 {|\phi(\zeta,w)|^{\mu+1/4+\alpha/2}|\zeta-z|^{2n-2\mu+\epsilon+\epsilon'}}dV(\zeta)
\\
 \lesssim \frac{|z-w|^{1/2+\alpha}}{|r(w)|^{1/2}}\gamma(w)
\end{multline*}
by (\ref{4ae}) and then consider
\begin{equation}
 \label{high}
\frac{|z-w|^{1/2+\alpha}}{|r(w)|^{1/2}}\gamma(w) \int_{D\cap
E_c\atop|\zeta-z|\le|\zeta-w|}
 \frac{1}
 {|\phi(\zeta,w)|^{\mu-1/4+\alpha/2}|\zeta-z|^{2n-2\mu+\epsilon+\epsilon'}}dV(\zeta).
\end{equation}
The case $\mu=1$ is trivial so we assume $\mu\ge2$ in which case
we use coordinates $s=-r(\zeta),t_1,\ldots,t_{2n-1}$ and bound
(\ref{high}) by
\begin{align*}
\frac{|z-w|^{1/2+\alpha}}{|r(w)|^{1/2}}\gamma(w) \int_V&
\frac{t^{2\mu-3-\epsilon-\epsilon'}}
 {(s+t^2)^{\mu-1/4+\alpha/2}}dsdt\\
 &\lesssim
 \frac{|z-w|^{1/2+\alpha}}{|r(w)|^{1/2}}\gamma(w)
 \int_V \frac{1}{s^{7/8}t^{3/4+\alpha+\epsilon+\epsilon'}} ds dt\\
&\lesssim
 \frac{|z-w|^{1/2+\alpha}}{|r(w)|^{1/2}}\gamma(w).
\end{align*}

To estimate $IV_b$ we use
 \begin{equation*}
\frac{|r(\zeta)|}{\gamma(\zeta)^2} \lesssim 1,
\end{equation*}
which follows by working in the coordinates of (\ref{rcoor}) near
a critical point, and thus
 we have
\begin{equation}
 \label{4bred}
IV_b\lesssim
 \int_{D\atop|\zeta-z|\le|\zeta-w|}
 \frac{\gamma(w)^{2-\epsilon'}}{\gamma(\zeta)^{\epsilon}}
 \frac{|z-w|}
 {|\phi(\zeta,w)|^{\mu+1}|\zeta-z|^{2n-2\mu}}dV(\zeta).
\end{equation}
We break the regions of integration in (\ref{4bred}) into
 $E_c$ and $D\setminus E_c$.  The estimates for $IV_b$ in the region
 $E_c$ are handled in the manner as was done for $IV_a$.
 In the region $D\setminus E_c$ we
 use $\gamma(w)\lesssim |\zeta-w|$ to bound (\ref{4bred}) by
\begin{align*}
\int_{D\setminus E_c \atop|\zeta-z|\le|\zeta-w|}&
\frac{\gamma(w)^{2-\epsilon'}}{\gamma(\zeta)^{\epsilon}}
 \frac{|z-w|}
 {|\phi(\zeta,w)|^{\mu+1}|\zeta-z|^{2n-2\mu}}dV(\zeta)\\
&\lesssim |z-w|^{\alpha}
 \int_{D\setminus E_c
\atop|\zeta-z|\le|\zeta-w|} \frac{1}{\gamma(\zeta)^{\epsilon}}
\frac{1}
 {|\phi(\zeta,w)|^{\mu-1/2+\epsilon'/2+\alpha/2}
 |\zeta-z|^{2n-2\mu}}dV(\zeta)\\
&\lesssim |z-w|^{\alpha}
 \int_{D\setminus E_c
\atop|\zeta-z|\le|\zeta-w|} \frac{1}{\gamma(\zeta)^{\epsilon}}
 \frac{1}
 {
 |\zeta-z|^{2n-1+\epsilon'+\alpha}}dV(\zeta)\\
&\lesssim |z-w|^{\alpha}.
\end{align*}

$ii)$. For $T^z$ a smooth first order tangential differential
operator on $D$, with respect to the $z$ variable, we have
\begin{align*}
T^zr&=0\\
T^zr^{\ast}&=\lre_{0,0}r\\
T^zP&=\lre_{1,0}+\lre_{0,0}\frac{r}
 {\gamma}\frac{r^{\ast}}{(\gamma^{\ast})^2}\\
&=\lre_{1,0} + \frac{\lre_{0,0}}{\gamma^{\ast}}(P+\lre_{2,0})\\
 T^z\phi&=\lre_{0,1}+\lre_{1,0}.
\end{align*}

We consider first the case in which the kernel of $A$ is of double
type $(1,3)$, of the form $\lra_{(3)}(\zeta,z)$,
  where the
subscript $(3)$ refers to the smooth type.

  Thus we write
\begin{equation}
 \label{typeaftertander}
\gamma^{\ast}T^z\lra_{(3)}= \gamma^{\ast}
\lra_{(1)}\gamma+\gamma^{\ast}\lra_{(2)} +\lra_{(3)},
\end{equation}
and estimate integrals involving the various forms the integral
kernels of different types assume.

We insert (\ref{typeaftertander}) into
\begin{equation*}
 \gamma^{\ast}TA_{(3)}f=\int_D f(\zeta)\gamma^{\ast}T^z
 \lra_{(3)}(\zeta,z)dV(\zeta)
\end{equation*}
 and we change
the factors of $\gamma^{\ast}$ through the equality
$\gamma(z)=\gamma(\zeta)+\lre_{1,0}$.  $ii)$ will then follow in
this case by the estimates
\begin{align}
\nonumber & \int_D
\frac{\gamma^{\epsilon'}(z)}{\gamma^{\epsilon}(\zeta)}
 |\lra_{(1)}(\zeta,z)|dV(\zeta)\lesssim
\frac{1}{|r(z)|^{\delta}}\\
\nonumber
 & \int_D
\frac{\gamma^{\epsilon'}(z)}{\gamma^{1+\epsilon}(\zeta)}
|\lra_{(2)}(\zeta,z)|dV(\zeta)\lesssim
\frac{1}{|r(z)|^{\delta}}\\
\label{smth3}
 & \int_D
\frac{\gamma^{\epsilon'}(z)}{\gamma^{2+\epsilon}(\zeta)}
|\lra_{(3)}(\zeta,z)|dV(\zeta)\lesssim \frac{1}{|r(z)|^{\delta}}.
\end{align}

We will prove the case of (\ref{smth3}) in which $\lra_{(3)}$
satisfies
\begin{equation*}
|\lra_{(3)}| \lesssim \frac{1}{P^{n-3/2-\mu}|\phi|^{\mu+1}} \qquad
\mu\ge 1.
\end{equation*}
The other cases are handled similarly.

Using the notation from $i)$ above, we choose coordinates
$u_{j_1},\ldots,u_{j_m},v_{j_{m+1}},\ldots,v_{j_{2n}}$ such
 that
 \begin{equation*}
 -r(\zeta)=u_{j_1}^2+\cdots+u_{j_m}^2-v_{j_{m+1}}^2-\cdots -
 v_{j_{2n}}^2,
 \end{equation*}
and let $U_{\varepsilon}=\bigcup_{j=1}^k U_{\varepsilon}(p_j)$.
We break the problem into subcases depending on whether $z\in
U_{\varepsilon}$.

Subcase $a)$.  Suppose $z\in U_{\varepsilon}(p_j)$.  We estimate
\begin{equation}
\label{Ujep} \int_{U_{2\varepsilon}(p_j)}
\frac{\gamma^{\epsilon'}(z)}{\gamma^{2+\epsilon}(\zeta)}
\frac{1}{|\phi|^{\mu+1}P^{n-3/2-\mu}} dV(\zeta)
\end{equation}
and
\begin{equation}
\label{Dminus}
 \int_{D_{\epsilon}\setminus U_{2\varepsilon}}
\frac{\gamma^{\epsilon'}(z)}{\gamma^{2+\epsilon}(\zeta)}
\frac{1}{|\phi|^{\mu+1}P^{n-3/2-\mu}} dV(\zeta).
\end{equation}

We break up the integral in (\ref{Ujep}) into integrals over
$E_c(z)$ and its complement, where $c$ is as in Lemma
\ref{schmalzlem}.  We also choose $c<1$ so that we also have the
estimate $|\zeta-z|\lesssim \gamma(\zeta)$.

We set $\theta=-r(z)$.

In the case $U_{2\varepsilon}(p_j)\cap E_c(z)$, we use a
coordinate system, $s=-r(\zeta)$, $t_1,\ldots,t_{2n-1}$, and
estimate
\begin{align*}
 \int_{ U_{2\varepsilon}(p_j)\cap E_c(z)}
 &\frac{\gamma^{\epsilon'}(z)}{\gamma^{2+\epsilon}(\zeta)}
\frac{1}{|\phi|^{\mu+1}P^{n-3/2-\mu}}
dV(\zeta)\\
\nonumber
 &\qquad
 \lesssim
 \int_{V}\frac{t^{2n-2}}
{\gamma^{1-\epsilon'}(z)(\theta+s+t^2)^{\mu+1}(s+t)^{2n-1-2\mu+\epsilon}}dsdt\\
 \nonumber
&\qquad
 \lesssim
 \int_{V}\frac{t^{2\mu-2+\epsilon'-\epsilon}}{(\theta+s+t^2)^{\mu+1}}dsdt\\
\nonumber
 &\qquad\lesssim
\frac{1}{\theta^{\delta}} \int_{V} \frac{t^{2\mu-2+\epsilon'-\epsilon}}{(s+t^2)^{\mu+1-\delta}} dsdt\\
 \nonumber
 &\qquad\lesssim\frac{1}{\theta^{\delta}}
 \int_0^M\frac{1}{s^{3/2-\delta}}ds\int_0^{\infty}
 \frac{\tilde{t}^{2\mu-2+\epsilon'-\epsilon}}{(1+\tilde{t}^2)^{\mu+1-\delta}}d\tilde{t}\\
 &\qquad\lesssim
  \frac{1}{\theta^{\delta}},
\end{align*}
where $M>0$ is some constant, and we make the substitution
$t=s^{1/2}\tilde{t}$.

We now estimate the integral
\begin{equation}
 \label{ujepset}
\int_{ U_{2\varepsilon}(p_j)\setminus E_c(z)}
\frac{\gamma^{\epsilon'}(z)}{\gamma^{2+\epsilon}(\zeta)}
\frac{1}{|\phi|^{\mu+1}P^{n-3/2-\mu}} dV(\zeta).
\end{equation}

Defining $u=\sqrt{u_{j_1}^2+\cdots+u_{j_m}^2}$,
$v=\sqrt{v_{j_{m+1}}^2+\cdots+v_{j_{2n}}^2}$, and using the
estimates from above
\begin{align*}
&|w(\zeta)|\lesssim |\zeta-z|\\
&|w(\zeta)|\lesssim \gamma(\zeta),
\end{align*}
where $w(\zeta)$ is defined as in (\ref{defnw}),
 we can bound the integral in
(\ref{ujepset}) by
\begin{align}
\label{inec}
 \int_{U_{2\varepsilon}(p_j)\setminus E_c(z)}
 &\frac{\gamma^{\epsilon'}(z)}{\gamma^{2+\epsilon}(\zeta)}
\frac{1}{|\phi|^{\mu+1}P^{n-3/2-\mu}}
dV(\zeta)\\
\nonumber
 &\qquad\lesssim
\int_V \frac{u^{m-1}v^{2n-m-1}}{
 (u+v)^{2n-1+\epsilon-\epsilon'}(\theta+u^2+v^2)} dudv\\
 \nonumber
 &\qquad\lesssim
 \int_V \frac{1}{(u+v)^{1+\epsilon-\epsilon'}(\theta+u^2+v^2)} dudv\\
 \nonumber
&\qquad\lesssim \frac{1}{\theta^{\delta}}
 \int_V \frac{1}{(u+v)^{3-2\delta+\epsilon-\epsilon'}} dudv\\
 \nonumber
&\qquad\lesssim \frac{1}{\theta^{\delta}}
 ,
\end{align}
where $V$ is a bounded region.  We have therefore bounded
(\ref{Ujep}), and we turn now to (\ref{Dminus}).

In $D\setminus U_{2\varepsilon}$ we have that $|\zeta-z|$ and
$\gamma(\zeta)$ are bounded from below so
\begin{equation*}
\int_{D\setminus U_{2\varepsilon}}
\frac{\gamma^{\epsilon'}(z)}{\gamma^{2+\epsilon}(\zeta)}
\frac{1}{|\phi|^{\mu+1}P^{n-3/2-\mu}} dV(\zeta)
 \lesssim  1.
\end{equation*}
This finishes subcase $a)$.

Case $b)$.  Suppose $z\notin U_{\varepsilon}$.  We divide $D$ into
the regions $D\cap E_c(z)$ and $D\setminus E_c(z)$.

In $D\cap E_c(z)$ the same coordinates and estimates work here as
in establishing the estimates for the integral in (\ref{inec}).

In $D\setminus E_c(z)$ we have $|\zeta-z|\gtrsim \gamma(z)$, but
$\gamma(z)$ is bounded from below, since $z\notin
U_{\varepsilon}$.  We therefore have to estimate
\begin{equation*}
\int_{D}\frac{1}{\gamma^{2+\epsilon}(\zeta)}dV(\zeta),
\end{equation*}
which is easily done by working with the coordinates
$w_1,\ldots,w_{2n}$ above.

 $iii)$.
The proof of $iii)$ follows the same steps as those in
 the proof of $ii)$, and we leave the details to the reader.
\end{proof}

\begin{thrm}
\label{e1commute}
 Let $X$ be a smooth tangential vector field.  Then
\begin{equation*}
 \gamma^{\ast}X^{z}E_{1-2n} =
 -E_{1-2n}\tilde{X}^{\zeta}\gamma+E_{1-2n}^{(0)}+\sum_{\nu=1}^l
E_{1-2n}^{(\nu)}   ,
\end{equation*}
where $\tilde{X}$ is the adjoint of $X$ and the $E_{1-2n}^{(\nu)}$
are isotropic operators.
\end{thrm}
\begin{proof}
The proof follows the line of argument used in proving case $1)$
of Theorem \ref{commutator}, and makes use of
 $(\gamma X^{\zeta}+\gamma^{\ast}X^{z})\lre_{1-2n}^i=\lre_{1-2n}^i$.
\end{proof}

\begin{thrm}
\label{E1properties}
 Let $T$ be a smooth tangential vector field.
 Set $E$ to be an operator with kernel of the form
 $\lre^i_{1-2n}(\zeta,z)R_1(\zeta)$ or $\lre^i_{2-2n}(\zeta,z)$ .
Then we have the following properties:
\begin{align*}
&i)\ E_{1-2n}:L^p(D)\rightarrow L^s(D)\\
&ii)\ E:L^{\infty,2+\epsilon,0}(D)\rightarrow
 \Lambda_{\alpha,2-\epsilon'}(D)
\qquad 0<\epsilon,\epsilon', \quad \alpha+\epsilon+\epsilon'<1\\
&iii)\ \gamma^{\ast}TE:\Lambda_{\alpha,2+\epsilon}(D)
\rightarrow L^{\infty,\epsilon',0}(D) \qquad \epsilon<\epsilon'\\
&iv)\ E:L^{\infty,\epsilon,\delta}(D)\rightarrow
 L^{\infty,\epsilon',0}(D) \qquad \epsilon<\epsilon',
 \quad \delta<1/2 +(\epsilon'-\epsilon)/2
\end{align*}
for any $1\le p\le s\le\infty$ with $1/s>1/p-1/2n$.
\end{thrm}
\begin{proof}
 $i)$ is presented in \cite{LiMi}.

 The proof of $ii)$ follows that of Theorem \ref{dertype1} $i)$.

For $iii)$ we let $\lre(\zeta,z)$ be the kernel of $E$, and we
calculate
\begin{align}
\nonumber
 (\gamma^{\ast})^{1+\epsilon'}TE  f =&\int_{D} f(\zeta)
\gamma^{\ast}T^z
 \lre(\zeta,z) dV(\zeta)\\
\nonumber
 =&\int_{D} (\gamma^{\ast})^{\epsilon'}\gamma^{2+\epsilon}f(\zeta)
\frac{\gamma^{\ast}T^z
 \lre(\zeta,z)}{\gamma^{2+\epsilon}} dV(\zeta)\\
\label{comint}
 =& \int_{D}  (\gamma^{\ast})^{\epsilon'}(\gamma^{2+\epsilon}f(\zeta)-(\gamma^{\ast})^{2+\epsilon}f(z))
\frac{\gamma^{\ast}T^z
 \lre(\zeta,z)}{\gamma^{2+\epsilon}}  dV(\zeta)\\
\nonumber
 &+(\gamma^{\ast})^{2+\epsilon}f(z)\int_{D}(\gamma^{\ast})^{\epsilon'}
 \frac{\gamma^{\ast}T^z\lre(\zeta,z)}{\gamma^{2+\epsilon}}
 dV(\zeta).
\end{align}
We use Theorem \ref{e1commute}
 in the last
integral to bound the last term of (\ref{comint}) by
\begin{align*}
(\gamma^{\ast})^{2+\epsilon}f(z)
 \int_{D} \lre_{1-2n,0} (\gamma^{\ast})^{\epsilon'}
 \left(\frac{1}{\gamma^{1+\epsilon}}+
\frac{\lre_{1,0}}{\gamma^{2+\epsilon}}
 \right)
 dV(\zeta)
  &\lesssim (\gamma^{\ast})^{2+\epsilon} |f(z)|\\
  &\lesssim \|  f \|_{L^{\infty,2+\epsilon,0}},\\
  &\lesssim \|  f \|_{\Lambda_{\alpha,2+\epsilon}},
\end{align*}
where the first inequality can be proved by breaking the integrals
into the regions $U_{2\varepsilon}$ and $D\setminus
U_{2\varepsilon}$ and in the region $D\setminus U_{2\varepsilon}$
using the same coordinates as in the proof of Theorem
\ref{dertype1} $ii)$.

For the first integral in (\ref{comint}), we
 note if $f\in\Lambda_{\alpha}$ then $\gamma^{2+\epsilon}f\in\Lambda_{\alpha}$.  We
have
\begin{align*}
 \label{ta1}
\int_{D} (\gamma^{\ast})^{\epsilon'}\Bigg|
\big(\gamma^{2+\epsilon}f(\zeta)-(\gamma^{\ast})^{2+\epsilon}&f(z)\big)
\frac{\gamma^{\ast} T^z
 \lre(\zeta,z)}{\gamma^{2+\epsilon}} \Bigg| dV(\zeta)\\
\nonumber
 &\lesssim \|\gamma^{2+\epsilon}f\|_{\Lambda_{\alpha}}
 \int_{D} |\zeta-z|^{\alpha}(\gamma^{\ast})^{\epsilon'}
 \left|\frac{\gamma^{\ast} T^z
 \lre(\zeta,z)}{\gamma^{2+\epsilon}}\right| dV(\zeta)\\
\nonumber &\lesssim  \|\gamma^{2+\epsilon}f\|_{\Lambda_{\alpha}}.
\end{align*}

The proof of $iv)$ follow as in the case of Theorem \ref{dertype1}
$iii)$.
\end{proof}

\section{ $C^k$ estimates}
\label{ck}

We define $Z_1$ operators to be those which take the form
\begin{equation*}
Z_1=A_{(1,1)}+E_{1-2n}\circ\gamma,
\end{equation*}
and we write Theorem \ref{basicintrep} as
\begin{equation}
\label{z1cal}
 \gamma^3 f =
 Z_1\gamma^2\mdbar f +Z_1\gamma^2\mdbar^{\ast} f
  +Z_1 f.
\end{equation}

We define $Z_j$ operators to be those operators of the form
\begin{equation*}
 Z_j=\overbrace{Z_1\circ\cdots\circ Z_1}^{j\mbox{ times}}.
\end{equation*}
  We establish
mapping properties for $Z_j$ operators
\begin{lemma}
\label{weightedz}
 For $1<p<\infty$ and $j\ge 1$
\begin{equation}
\label{jnorm}
 \| Z_j f\|_p\lesssim \|  f\|_{L^{p,jp}},
\end{equation}
and  for $0<\epsilon'<\epsilon$
 \begin{equation}
\label{inftynorm}
 \| Z_j f\|_{L^{\infty,\epsilon,0}}\lesssim \|
f\|_{L^{\infty,j+\epsilon',0}}.
\end{equation}
\end{lemma}
\begin{proof}
 We prove (\ref{jnorm}) for kernels of the form
 $\lra_{(1,1)}(\zeta,z)$, where $\lra_{(1,1)}$ is a kernel of
  double type $(1,1)$.
 We show below that $\lra_{(1,1)}(\zeta,z)$ satisfies
\begin{equation}
\label{aeproperty}
 \sup_{z\in\Omega}\int \frac{1}{\gamma(\zeta)}
 |\lra_{(1,1)}(\zeta,z)|
 |r(\zeta)|^{-\delta}|r(z)|^{\delta}dV(\zeta) <\infty
\end{equation}
for $\delta<1$.  The lemma then follows from the generalized
Young's inequality.

We further restrict our proof to the cases in which $\lra_{(1,1)}$
satisfies
\begin{align*}
& i)\ \frac{1}{\gamma(\zeta)}|\lra_{(1,1)}| \lesssim
\gamma(\zeta)\frac{1}{P^{n-1/2-\mu}|\phi|^{\mu+1}}
\qquad \mu\ge 1 \\
& ii) \frac{1}{\gamma(\zeta)}|\lra_{(1,1)}| \lesssim
\frac{1}{P^{n-1-\mu}|\phi|^{\mu+1}} \qquad \mu\ge 1\\
& iii) \frac{1}{\gamma(\zeta)}|\lra_{(1,1)}| \lesssim
\frac{1}{\gamma(\zeta)}\frac{1}{P^{n-3/2-\mu}|\phi|^{\mu+1}}
\qquad \mu\ge 1
\end{align*}
We will prove the more difficult case $iii)$, as cases $i)$ and
$ii)$ follow similar arguments, and we leave the details of those
cases to the reader.

We use the same notation as in Theorem \ref{dertype1} $iii)$.  As
in Theorem \ref{dertype1} $iii)$ we divide the estimates into
subcases depending on whether $z\in U_{\varepsilon}$.

Subcase $a)$.  Suppose $z\in U_{\varepsilon}(p_j)$.  We estimate
\begin{equation}
\label{zUjep} \int_{U_{2\varepsilon}(p_j)}
\frac{1}{\gamma(\zeta)|\phi|^{\mu+1}P^{n-3/2-\mu}|r(\zeta)|^{\delta}}
dV(\zeta)
\end{equation}
and
\begin{equation}
\label{zDminus}
 \int_{D_{\epsilon}\setminus U_{2\varepsilon}}
\frac{1}{\gamma(\zeta)|\phi|^{\mu+1}P^{n-3/2-\mu}|r(\zeta)|^{\delta}}
dV(\zeta).
\end{equation}

We break up the integral in (\ref{zUjep}) into integrals over
$E_c(z)$ and its complement, where $c$ is as in Lemma
\ref{schmalzlem}, and we choose $c<1$.  Thus, in $E_c(z)$, we have
$|\zeta -z|\lesssim \gamma(\zeta)$.

In the case $U_{2\varepsilon}(p_j)\cap E_c(z)$, we use a
coordinate system, $s=-r(\zeta)$, $t_1,\ldots,t_{2n-1}$, and
estimate
\begin{align}
\label{z1stcoor}
 \int_{ U_{2\varepsilon}(p_j)\cap E_c(z)}
 &\frac{1}{\gamma(\zeta)|\phi|^{\mu+1}P^{n-3/2-\mu}|r(\zeta)|^{\delta}}
dV(\zeta)\\
\nonumber
 &\qquad
 \lesssim
 \int_{\mathbb{R}_+^2}\frac{t^{2n-2}}{\gamma(\zeta)\gamma(z)s^{\delta}(\theta+s+t^2)^{\mu+1}(s+t)^{2n-3-2\mu}}dsdt\\
 \nonumber
&\qquad
 \lesssim
 \int_{\mathbb{R}_+^2}\frac{t^{2\mu-1}}{s^{\delta}(\theta+s+t^2)^{\mu+1}}dsdt\\
\nonumber
 &\qquad\lesssim
 \int_{\mathbb{R}_+^2} \frac{1}{s^{\delta}(\theta^{1/2}+s^{1/2}+t)^{3}} dsdt\\
 \nonumber
 &\qquad\lesssim
 \int_0^{\infty} \frac{1}{s^{\delta}(\theta+s)} ds\\
 \nonumber
 &\qquad\lesssim \frac{1}{\theta^{\delta}},
\end{align}
where we use the notation
 $\mathbb{R}_+^j=\overbrace{\mathbb{R}_+\times\cdots\times\mathbb{R}_+}^{j\mbox{ times}}$.

We now estimate the integral
\begin{equation}
 \label{zujepset}
\int_{ U_{2\varepsilon}(p_j)\setminus E_c(z)}
\frac{1}{\gamma(\zeta)|\phi|^{\mu+1}P^{n-3/2-\mu}|r(\zeta)|^{\delta}}
dV(\zeta).
\end{equation}

Recall from above that with the coordinates
 $u_{j_1},\ldots,u_{j_m},v_{j_{m+1}},\ldots,v_{j_{2n}}$
 so that around the critical point, $p_j$ we have
\begin{equation*}
-r(\zeta)=u_{j_1}^2+\cdots+u_{j_m}^2-v_{j_{m+1}}^2-\cdots-v_{j_{2n}}^2
\end{equation*}
 and with $w_1,\ldots,w_{2n}$ defined by
 \begin{equation*}
 w_{\alpha}=\begin{cases}
 u_{j_{\alpha}}\quad \mbox{for } 1\le
\alpha\le m\\
 v_{j_{\alpha}} \quad \mbox{for } m+1\le\alpha\le 2n,
\end{cases}
\end{equation*}
 we have
$ |w(\zeta)|\lesssim|\zeta-z| $ and $|w(\zeta)|\lesssim
\gamma(\zeta)$ for $\zeta,z\in U_{2\varepsilon}(p_j)$.

We can therefore bound the integral in (\ref{zujepset}) by
\begin{align}
\nonumber
 \int_{U_{2\varepsilon}(p_j)\setminus E_c(z)}
 &\frac{1}{\gamma(\zeta)|\phi|^{\mu+1}P^{n-3/2-\mu}|r(\zeta)|^{\delta}}
dV(\zeta)\\
\nonumber
 &\qquad\lesssim
\int_V \frac{u^{m-1}v^{2n-m-1}}{
 (u+v)^{2n-2}(\theta+u^2+v^2)(u^2-v^2)^{\delta}} dudv\\
 \label{casedel}
 &\qquad\lesssim
 \int_V \frac{1}{(\theta+u^2)(u^2-v^2)^{\delta}} dudv,
\end{align}
where $V$ is a bounded region.  We make the substitution
$v=\tilde{v}u$, since $v^2<u^2$, and write (\ref{casedel}) as
\begin{align*}
\int_0^{M}\frac{1}{u^{2\delta-1}(\theta+u^2)}du
  \int_0^1\frac{1}{(1-\tilde{v}^2)^{\delta}}d\tilde{v}
   &\lesssim\frac{1}{\theta^{\delta}}\int_0^M
    \frac{1}{u^{2\delta-1}(1+u^2)}du\\
   &\lesssim\frac{1}{\theta^{\delta}},
\end{align*}
where $M>0$ is some constant.  We have therefore bounded
(\ref{zUjep}), and we turn now to (\ref{zDminus}).

In $D\setminus U_{2\varepsilon}$ we have that $|\zeta-z|$ and
$\gamma(\zeta)$ are bounded from below so
\begin{equation*}
\int_{D\setminus U_{2\varepsilon}}
\frac{1}{\gamma(\zeta)|\phi|^{\mu+1}P^{n-3/2-\mu}|r(\zeta)|^{\delta}}
dV(\zeta)
 \lesssim \int_{D\setminus U_{2\varepsilon}}
\frac{1}{|r(\zeta)|^{\delta}} dV(\zeta)\lesssim 1,
\end{equation*}
the last inequality following because in $D\setminus
U_{2\varepsilon}$ $r$ can be chosen as a coordinate since
$\gamma(\zeta)$ is bounded from below.  This finishes subcase
$a)$.

Subcase $b)$.  Suppose $z\notin U_{\varepsilon}$.  We divide $D$
into the regions $D\cap E_c(z)$ and $D\setminus E_c(z)$.

In $D\cap E_c(z)$ the same coordinates and estimates work here as
in establishing the estimates for the integral in
(\ref{z1stcoor}).

In $D\setminus E_c(z)$ we have $|\zeta-z|\gtrsim \gamma(z)$, but
$\gamma(z)$ is bounded from below, since $z\notin
U_{\varepsilon}$.  We therefore have to estimate
\begin{equation*}
\int_{D}\frac{1}{\gamma(\zeta)|r(\zeta)|^{\delta}}dV(\zeta),
\end{equation*}
which is easily done by working with the coordinates
$w_1,\ldots,w_{2n}$ above.

(\ref{inftynorm}) is proved similarly.
\end{proof}

\begin{lemma}
\label{A1property}
 Let $T$ be a tangential vector field and $\varepsilon>0$.  For
 $\epsilon>0$ sufficiently small
\begin{align*}
&i)\ Z_{n+2}:L^2(D)\rightarrow L^{\infty}(D) \qquad  \\
&ii)\ \| \gamma T Z_4 f\|_{C^{1/4-\varepsilon}} \lesssim \|
f\|_{L^{\infty,3+\epsilon,0}}
\end{align*}
\end{lemma}
\begin{proof}
For $i)$ apply Corollary \ref{yngcor} and Theorem
\ref{E1properties} $i)$, $n+2$ times.

 For $ii)$ we let $\alpha<1/4$, and apply the commutator theorem, Theorem \ref{commutator},
  and consider the two
compositions
 $Z_1\circ Z_1\circ\gamma T A_1\circ Z_1$, and
$Z_1\circ Z_1\circ\gamma T E\circ Z_1$.    From Theorems
\ref{dertype1} and \ref{E1properties} we can find
$\epsilon_1,\ldots,\epsilon_4$ such that
$0<\epsilon_{j+1}<\epsilon_j$ and such that in
 the first case we have
\begin{multline*}
\|Z_1\circ Z_1\circ \gamma T A_1\circ Z_1
f\|_{\Lambda_{\alpha}}\lesssim
\|Z_1\circ \gamma T A_1\circ Z_1 f\|_{L^{\infty,\epsilon_1,0}}\lesssim\\
  \|\gamma
T A_1\circ Z_1 f\|_{L^{\infty,\epsilon_2,\delta}}\lesssim
 \|Z_1f\|_{L^{\infty,2+\epsilon_3,0}}\lesssim \|f\|_{L^{\infty,3+\epsilon_4,0}}
\end{multline*}
and, in the second,
\begin{multline*}
 \|Z_1\circ Z_1\circ \gamma T E\circ Z_1 f\|_{\Lambda_{\alpha}}\lesssim
\|Z_1\circ \gamma T E\circ Z_1
f\|_{L^{\infty,\epsilon_1,0}}\lesssim
\\ \|\gamma
T E\circ Z_1 f\|_{L^{\infty,1+\epsilon_2,0}}\lesssim
 \|Z_1f\|_{\Lambda_{\alpha,3+\epsilon_3}}\lesssim \|f\|_{L^{\infty,3+\epsilon_4,0}},
\end{multline*}
where the second and third inequalities are proved in the same way
as Theorem \ref{E1properties} $ii)$ and $iii)$.
\end{proof}

 We now iterate (\ref{z1cal}) to get
\begin{align}
\label{basiciterate}
 \gamma^{3j}f=&(Z_1\gamma^{3(j-1)+2}+Z_2\gamma^{3(j-2)+2}+
  \cdots+Z_{j}\gamma^2)\mdbar f \\
 \nonumber
  &+(Z_1\gamma^{3(j-1)+2}+Z_2\gamma^{3(j-2)+2}+
  \cdots+Z_{j}\gamma^2)\mdbar^{\ast}f+Z_{j}f.
\end{align}
Then we can prove
\begin{thrm} For $f\in L^2_{0,q}(D)\cap\mbox{Dom}(\mdbar)\cap
\mbox{Dom}(\mdbar^{\ast})$, $q\ge 1$, and $\varepsilon>0$
\label{k=0}
\begin{equation*}
\|\gamma^{3(n+3)}f\|_{C^{1/4-\varepsilon}}\lesssim
\|\gamma^2\mdbar
f\|_{\infty}+\|\gamma^2\mdbar^{\ast}f\|_{\infty}+\|f\|_2.
\end{equation*}
\end{thrm}
\begin{proof}
Use Theorems \ref{dertype1} $i)$ and \ref{E1properties} $ii)$ and
Lemma \ref{A1property} $i)$ in (\ref{basiciterate}) with $j=n+3$
\end{proof}

We use the notation $D^k$ to denote a $k^{th}$ order differential
operator, which is a sum of terms which are composites of $k$
 vector fields.

 We define
\begin{equation*}
Q_k(f)=\sum_{j=0}^k\|\gamma^{j+2} D^j \mdbar
f\|_{\infty}+\sum_{j=0}^k\|\gamma^{j+2}
D^j\mdbar^{\ast}f\|_{\infty}+\|f\|_{2}.
\end{equation*}

$T^k$ will be used for  a $k$-th order tangential differential
operator, which is a sum of terms which are composites of $k$
 tangential vector fields.

\begin{lemma}
\label{tanglemma}
  Let $T^k$ be a tangential operator of order
$k$. For $\varepsilon, \epsilon >0$
\begin{equation*}
\|\gamma^{3(n+6)+8k+\epsilon}T^k f\|_{C^{1/4-\varepsilon}}\lesssim
Q_k(f).
\end{equation*}
\end{lemma}
\begin{proof}
 We first prove
\begin{equation}
 \label{inftyinduct}
\|\gamma^{3(n+2)+9+8k+\epsilon}T^k f\|_{L^{\infty}}\lesssim
Q_k(f).
\end{equation}
The proof is by induction in which the first step is proved as was
Theorem \ref{k=0}.
 We choose $j=3$ in (\ref{basiciterate}) and then apply (\ref{basiciterate})
to $\gamma^{3(n+2)+7k}f$ to get
\begin{equation*}
\gamma^{3(n+2)+9+7k}f=
 Z_1\gamma^2\mdbar f
  \nonumber
 +Z_1\gamma^2\mdbar^{\ast} f
  +Z_{3} \gamma^{3(n+2)+7k}f.
\end{equation*}
We then apply $\gamma^{\epsilon}(\gamma T)^k$, where $T$ is a
tangential operator. We use the commutator theorem, Theorem
\ref{commutator}, to show
\begin{align}
\label{kderinfty}
  \gamma^{3(n+2)+9+8k+\epsilon}T^k f=& \gamma^{\epsilon}\sum_{j=0}^{k-1}
Z_{3}
 \gamma^{3(n+2)+7k+j}T^jf
+ \gamma^{\epsilon}\gamma T Z_{3}
 \gamma^{3(n+2)+8k-1}T^{k-1}f\\
\nonumber
 &+
 \gamma^{\epsilon}\sum_{j=0}^k
  Z_1\gamma^{j+2}
 T^j
 \mdbar f
 +
  \gamma^{\epsilon}\sum_{j=0}^k
  Z_1\gamma^{j+2}
 T^j
 \mdbar^{\ast} f.
\end{align}
By Lemma \ref{weightedz} and the induction hypothesis, we conclude
the $L^{\infty}$ norm of the first term on the right hand side of
(\ref{kderinfty}) is bounded by $Q_{k-1}(f)$.

In the same way we proved Lemma \ref{A1property}, we have
\begin{equation*}
\gamma T Z_3:L^{\infty,3+\epsilon',0}(D)\rightarrow
L^{\infty,\epsilon,0}(D),
\end{equation*}
for some $0<\epsilon'<\epsilon$ and so the $L^{\infty}$ norm of
the second term is bounded by
\begin{equation*}
\|\gamma^{3(n+2)+8k+2+\epsilon'}T^{k-1}f\|_{L^{\infty}}
 \lesssim \|\gamma^{3(n+2)+9+8(k-1)}T^{k-1}f\|_{L^{\infty}}\lesssim
Q_{k-1}(f).
\end{equation*}
The last two terms on the right side of (\ref{kderinfty}) are
obviously bounded by $Q_{k}(f)$, and thus we are done with the
proof of (\ref{inftyinduct}).

To finish the proof of the lemma, we follow the proof of
(\ref{inftyinduct}), and choose $k=4$ in (\ref{basiciterate}),
then apply (\ref{basiciterate}) to $\gamma^{3(n+2)+7k}f$, and
again apply the operators $\gamma^{\epsilon}(\gamma T)^k$, where
$T$ is a tangential operator. In this way, we show
\begin{align}
\label{kder} \gamma^{3(n+2)+12+8k+\epsilon}T^k f=&
\gamma^{\epsilon}\sum_{j=0}^{k-1} Z_{4}
 \gamma^{3(n+2)+7k+j}T^jf
+ \gamma^{\epsilon}\gamma T Z_{4}
 \gamma^{3(n+2)+8k-1}T^{k-1}f\\
\nonumber
 &+
 \gamma^{\epsilon}\sum_{j=0}^k
  Z_1\gamma^{j+2}
 T^j
 \mdbar f
 +
  \gamma^{\epsilon}\sum_{j=0}^k
  Z_1\gamma^{j+2}
 T^j
 \mdbar^{\ast} f.
\end{align}
By Theorems \ref{dertype1} $i)$ and \ref{E1properties} $ii)$, for
some $\epsilon'>0$, the first sum on the right hand side of
(\ref{kder}) has its $C^{1/4-\varepsilon}$ norm bounded by
\begin{equation*}
\| Z_3 \gamma^{3(n+2)+7k+\epsilon'+j}T^jf\|_{L^{\infty}}\lesssim
Q_{k-1}(f)
\end{equation*}
from above.  We can use Lemma \ref{A1property} $ii)$ to show the
$C^{1/4-\varepsilon}$ norm of the second term is bounded by
\begin{equation*}
\|\gamma^{3(n+2)+10+8(k-1)+\epsilon'}T^{k-1}f\|_{\infty} \lesssim
Q_{k-1}(f)
\end{equation*}
as above.

The last two terms on the right hand side of (\ref{kder}) are
easily seen to be bounded by $Q_k(f)$, and this finishes Lemma
\ref{tanglemma}.
\end{proof}
In order to generalize Lemma \ref{tanglemma} to include
non-tangential operators, we use the familiar argument of
utilizing the ellipticity of
 $\mdbar\oplus\mdbar^{\ast}$ to express a normal derivative of a component
 of a $(0,q)$-form, $f$, in terms of
tangential operators acting on components of $f$ and components of
$\mdbar f$ and $\mdbar^{\ast}f$.
 With the $(0,q)$-form $f$ written
\begin{equation*}
f=\sum_{|J|=q}f_J\bar{\omega}^J
\end{equation*}
locally,
 we have the decomposition in the following form:
\begin{equation}
 \label{normaldecomp}
\gamma N f_J=\sum_{jK}a_{JjK}\gamma T_jf_K +
    \sum_L b_{JL}f_L+ \sum_M c_{JM}\gamma(\mdbar f)_M
+ \sum_P d_{JP}\gamma(\mdbar^{\ast}f)_P,
\end{equation}
where $N=L_n+\bar{L}_n$ is the normal vector field, and
$T_1,\ldots, T_{2n-1}$ are the tangential fields as described in
Section \ref{est}.  The coefficients $a_{JjK}$, $b_{JL}$,
$c_{JM}$, and $d_{JP}$ are all of the form $\lre_{0,0}$ and the
index sets are strictly ordered with $J,K,L,M,P\subset
\{1,\ldots,n\}$, $|J|=|K|=|L|=q$, $|M|=q+1$, $|P|=q-1$,
$j=1,\ldots, 2n-1$.  The decomposition is well known in the smooth
case (see \cite{LiMi}) and to verify (\ref{normaldecomp}) in a
neighborhood of $\gamma=0$, one may use the coordinates
$u_{j_1},\ldots,u_{j_m},v_{j_{m+1}},\ldots,v_{j_{2n}}$ as in
(\ref{rcoor}) above.  For instance, integrating by parts to
compute $\mdbar^{\ast} f$ leads to terms of the form
$\lre_{0,-1}f_J$, whereby multiplication by $\gamma$ allows us to
absorb these terms into $b_{JL}$.

It is then straightforward how to generalize Lemma
\ref{tanglemma}.   Suppose $D^k$ is a $k^{th}$ order differential
operator which contains the normal field at least once.   In
$\gamma^kD^k$ we commute $\gamma N$ with terms of the form $\gamma
T$, where $T$ is tangential, and we consider the operator $D^k=
D^{k-1}\circ \gamma N$, where $D^{k-1}$ is of order $k-1$.  The
error terms due to the commutation involve differential operators
of order $\le k-1$.  From (\ref{normaldecomp}) we just have to
consider $D^{k-1}\gamma T f$, $D^{k-1}\mdbar f$, and
$D^{k-1}\mdbar^{\ast} f$.  The last two terms are bounded by
$Q_{k-1}(f)$, and we repeat the process with $D^{k-1}\gamma T f$,
until we are left with $k$ tangential operators for which we can
apply Lemma \ref{tanglemma}.

We thus obtain the weighted $C^k$ estimates
\begin{thrm}
 Let $f\in L^2_{0,q}(D)\cap\mbox{Dom}(\mdbar)\cap
\mbox{Dom}(\mdbar^{\ast})$, $q\ge 1$, $\alpha<1/4$, and
$\epsilon>0$. Then
\begin{equation*}
\|\gamma^{3(n+6)+8k+\epsilon} f\|_{C^{k+\alpha}}\lesssim Q_k(f).
\end{equation*}
\end{thrm}
As an immediate consequence we obtain weighted $C^k$ estimates for
the canonical solution to the $\mdbar$-equation.
\begin{cor}
Let $q\ge 2$ and let $N_q$ denote the $\mdbar$-Neumann operator
for $(0,q)$-forms. Let $f$ be a $\mdbar$-closed $(0,q)$-form. Then
for $\alpha<1/4$ and $\epsilon>0$, the canonical solution,
$u=\mdbar^{\ast}N_q f$ to $\mdbar u=f$, satisfies
\begin{equation*}
\|\gamma^{3(n+6)+8k+\epsilon} u\|_{C^{k+\alpha}}\lesssim
 \| \gamma^{k+2} f\|_{C^{k}} + \|f\|_2.
\end{equation*}
\end{cor}

\end{document}